\documentclass[12pt]{conm-p-l}   
\usepackage{amscd}
\usepackage{amsmath}
\usepackage{latexsym}
\usepackage{graphicx}
\usepackage{amsfonts}
\usepackage{amssymb}
\usepackage[all]{xy}
\newcommand{\edit}[1]{\marginpar{\footnotesize{#1}}}

\newcommand{\rank}{\mathop{\mathrm{rank}}}
\newcommand{\disc}{\mathop{\mathrm{disc}}}

\newcommand{\bP}{\mathbb{P}}
\newcommand{\R}{\mathbb{R}}
\newcommand{\C}{\mathbb{C}}
\renewcommand{\H}{\mathbb{H}}
\newcommand{\Q}{\mathbb{Q}}
\newcommand{\Z}{\mathbb{Z}}
\newcommand{\GW}{\mathbb{G}W}

\newcommand{\XG}{X_G}   

\newcommand{\cA}{\mathcal{A}}

\newcommand{\cH}{\mathcal{H}}
\newcommand{\cO}{\mathcal{O}}
\newcommand{\pt}{{\mathrm{pt}}}

\newcommand{\Hom}{\operatorname{Hom}}

\newcommand{\an}{\mathrm{an}}
\newcommand{\et}{\mathrm{et}}
\newcommand{\onto}{{\mathrm{onto}}}
\renewcommand{\top}{{\mathrm{top}}}
\newcommand{\tors}{\mathrm{tors}}
\newcommand{\tH}{\widetilde{H}}

\numberwithin{equation}{section}
\theoremstyle{plain}
\newtheorem{theorem}[equation]{Theorem}

\newtheorem{corollary}[equation]{Corollary}
\newtheorem{proposition}[equation]{Proposition}
\newtheorem{lemma}[equation]{Lemma}
\newtheorem{substuff}{\bf Remark}[equation] 

\theoremstyle{definition}

\newtheorem{definition}[equation]{Definition}
\newtheorem{examples}[equation]{Examples}
\newtheorem{example}[equation]{Example}
\newtheorem{subex}[substuff]{Example}
\newtheorem{subexamples}[substuff]{Examples}

\newtheorem{defn}[equation]{Definition}

\theoremstyle{remark}
\newtheorem{remark}[equation]{Remark}

\newtheorem{porism}[equation]{Porism}

\newtheorem{subremark}[substuff]{Remark} 

\newtheorem*{notations}{Notation}

\def\smap#1{\ {\buildrel #1 \over \rightarrow}\ }
\def\lmap#1{\ {\buildrel #1 \over \leftarrow}\ }
\def\map#1{{\ \buildrel #1 \over \longrightarrow}\ }
\newcommand{\mathdot}{{\mathbf{\scriptscriptstyle\bullet}}}
\newcommand{\Spec}{\operatorname{Spec}}
\newcommand{\Pic}{\operatorname{Pic}}
\newcommand{\Br}{\operatorname{Br}}
\newcommand{\RX}{X\!\times\!\R} 

\newcommand{\Gm}{\mathbb{G}_{m}}
\newcommand{\VR}{GR^{[1]}}
\newcommand{\uu}[1]{GR^{[-1]}_{#1}}

\begin{document}
\title[Witt group of real surfaces]{The Witt group of real surfaces}
\date{\today}

\author{Max Karoubi}
\address{Universit\'e Denis Diderot Paris 7 \\
Institut Math\'ematique de Jussieu --- Paris Rive Gauche, FRANCE}
\email{max.karoubi@gmail.com}
\urladdr{http://webusers.imj-prg.fr/~max.karoubi}

\author{Charles Weibel}
\address{Math.\ Dept., Rutgers University, New Brunswick, NJ 08901, USA}
\email{weibel@math.rutgers.edu}\urladdr{http://math.rutgers.edu/~weibel}
\thanks{Weibel was supported by NSA and NSF grants}

\begin{abstract}
Let $V$ be an algebraic variety defined over $\R$, and 
$V_\top$ the space of its complex points. We compare the algebraic 
Witt group $W(V)$ of symmetric bilinear forms on vector bundles over $V$, 
with the topological Witt group $WR(V_{\top})$ of
symmetric forms on Real vector bundles over $V_\top$
in the sense of Atiyah, 
especially when $V$ is 2-dimensional. 
To do so, we develop topological tools to calculate $WR(V_\top)$, and 
to measure the difference between $W(V)$ and $WR(V_\top)$.
\end{abstract}
\maketitle

\pagestyle{myheadings} 
\setcounter{section}{0}%

If $V$ is a smooth algebraic surface defined over the real numbers $\R$,
the Witt group $W(V)$ is a finitely generated abelian group, whose
rank equals the number $\nu$ of components of the 2-manifold $V(\R)$
\cite{Mahe}. Sujatha has given formulas for the torsion subgroup
in \cite{Sujatha}; although the formulas she gives are not always easy 
to compute, they involve \'etale cohomology and are thus
mostly topological in nature. Indeed, Cox' theorem \cite{Cox} states that
$H^n_\et(V,\Z/2)\cong \H_G^n(V_\top,\Z/2)$.
Here $V_\top$ is the 4-manifold underlying the complex variety 
$V\times_{\R}\C$, $G$ is the cyclic group $G=\text{Gal}(\C/\R)$ of order~2, 
acting on $V_\top$ by complex conjugation, and $\H_G^*$ is Borel cohomology.

Together with Schlichting, we introduced the topological
Witt group $WR(X)$ of a $G$-space $X$ in \cite{KSW}
(see Definition \ref{def:WR} below),
and showed that the natural map $W(V)\to WR(V_\top)$ is always
an isomorphism modulo bounded 2-primary torsion; for curves it is an
isomorphism; for surfaces,
the kernel and cokernel of the map $W(V)\to WR(V_\top)$ 
have exponent~2 by \cite[Thm.\,8.7]{KSW}.  
%

One of the main goals of this paper is to develop topological tools 
to compute $WR(X)$ when $X$ has dimension $\le4$, since if $V$ is
an algebraic surface then $WR(V_\top)$ is very close to $W(V)$.
The second goal is to more precisely measure the difference
between these two invariants in various examples.

Here are three important invariants of $W$ that have
analogues for $WR$. One is the classical signature on $W(V)$,
which factors through the topological signature on $WR(V_\top)$,
which is an invariant defined on the 1--skeleton of $X$ for any $G$-space $X$
(see Lemma \ref{signature}).
Another is the algebraic discriminant; its analogue defined on $WR(X)$
takes values in the group $\Pic_G(X)$ of equivariant real line bundles;
see Definition \ref{def:PicG}. 
The third is the Hasse invariant 
(see Theorem \ref{thm:Hasse}.) 

For any smooth projective surface $V$\!, $WR(V_\top)$ is a 
birational invariant of $V$ (see Theorem \ref{birational}).
If $V$ is either a complex projective surface,
so that $V_\top\!=\!G\!\times\!V(\C)$,
or a surface defined over $\R$ with no real points, 
(so that $V_\top$ is connected) the main theorems of \cite{RGB}
(see \ref{pg=0} and \ref{WR:pg=0} below; cf.\,\cite{Zibrowius})
state that $W(V)\!\to\!WR(V_\top)$ is a split surjection,
and an isomorphism if and only if the geometric genus 
$p_g=\dim H^0_\an(X,\Omega^2_X)$ is $0$.
In particular, when $V$ is a rational surface with $G$ acting freely
then $W(V)\cong WR(V_\top)$.  To compute this group it suffices to
compute $WR(X)$ when $X=S^2\times S^2$. This is done 
in Theorem \ref{formsof P1xP1}; there are 3 cases, corresponding
to the 3 possible actions of $G$ on $H^2(X,\Z)$. 

When $G$ does not act freely, we determine the kernel and cokernel
of $W(V)\to WR(V_\top)$ in Theorem \ref{thm:W-WR}; 
the kernel is 0 if and only if $p_g(V)=0$. 
We also show in Theorem \ref{WR-X/G} that when
$H_G^3(X;\Z(1))$ has no 2-torsion and $X^G$ has 
$\nu>0$ connected components then
\[  WR(X) \cong \Z^\nu \oplus H^1(X/G,\Z/2).  \]

The tools used to compute $WR(X)$ are markedly different from the tools used
to compute $W(V)$, because $WR$ is related to Atiyah's Real $K$-theory and to
the equivariant $K$-group $KO_G$. 
Both $KR$ and $KO_G$ satisfy
Bredon's axioms \cite{Bredon} for a $G$-equivariant cohomology theory:
homotopy invariance, excision, and long exact sequences.
In the Appendix, 
we recall the facts we need about 
the Bredon spectral sequence converging to equivariant cohomology.

\smallskip
Our paper is organized as follows.
In Section \ref{sec:W}, we review Sujatha's results on the Witt group
of a surface, as a reference for our results.  
In Section \ref{sec:WR}, we define $WR(X)$ and review the construction 
of the signature map. In Section \ref{sec:lowdim}, we compute
$WR(X)$ when $\dim(X)\le2$.
We define the topological determinant in Section \ref{sec:Pic_G},
relating it to the group $\Pic_G(X)$ of symmetric forms on 
Real line bundles.

In Section \ref{sec:birational},
we show that $WR$ is a birational invariant of projective surfaces over $\R$.
In Section \ref{sec:SW}, we use the Stiefel--Whitney classes
$w_1$ and $w_2$ to define maps on $WR(X)$ compatible with the 
discriminant and Hasse invariant on $W(V)$.

In Section \ref{sec:C}, we recall and illustrate the 
calculations of \cite{RGB}
for $WR(X)$ when $V$ has no real points.
In Section \ref{sec:R-points} we study the map
$W(V)\to WR(X)$ when $V$ has real points. 

In Section \ref{sec:3-folds}, we say a few things about
3--folds.\ The group $WR(V_\top)$ is easy to describe for 
complex 3--folds, and not much harder for 3--folds over $\R$
which have no real points. 
We did not look very hard at 3-folds with real points, because of 
Parimala's result in \cite{Parimala},
showing that for a smooth algebraic 3-fold that
$W(V)$ is finitely generated if and
only if the mod-2 Chow group $CH^2(V)/2$ is finitely generated.
(We give the short proof in Theorem \ref{3folds-fg} below.)
Totaro has shown in \cite{Totaro} that $CH^2(V)/2$ is in
fact not finitely generated for a very general abelian 3-fold,
making computations of $W(V)$ for 3-folds problematic.

The Fundamental Theorem for Witt theory says that 
$W(V\times\Gm)\cong W(V)\oplus W(V)$ for a smooth variety $V$. 
This is true more generally when $K_{-1}(V)=0$ 
(which is always true for smooth $V$); see \cite{KSW1}.
In Section \ref{sec:Gm} of the present paper, we show that 
if $KR_{-1}(X)=0$ then $WR(X\times S^{1,1})\cong WR(X)\oplus WR(X)$
where $ S^{1,1}$ is the circle $(\Gm)_\top$. 

\begin{notations}
Recall that $G$ denotes $\Z/2$.
If $X$ is a space with involution, we write $\H_G^*(X,A)$
for the Borel cohomology of $X$ with coefficients in a $G$-module $A$.
Similarly, we write $H_G^*(X,h)$ for the Bredon cohomology of $X$, 
where $h$ is a coefficient system; see the Appendix.
We will occasionally write $H_\et(V)$ for $H_\et(V,\Z/2)$.
\end{notations}

\smallskip
\paragraph{\em Acknowledgements:}
The authors thank the referee, and are grateful to
Parimala, R.\ Sujatha and J.-L.\ Colliot-Th\'el\`ene 
for several discussions. We are also grateful to Marco Schlichting
for several discussions about Section \ref{sec:Gm}.

\smallskip\goodbreak
\section{A filtration on $W(V)$}\label{sec:W}

Let $k$ be a field containing $1/2$.
A  {\it symmetric form} $(E,\theta)$ on an algebraic variety $V$ over $k$
is an algebraic vector bundle $E$ on $V$, with an isomorphism $\theta$ from $E$
to its dual $E^*$ such that $\theta=\theta^*$. Asssociated to any bundle $E$,
there is a hyperbolic form $H(E)=(E\oplus E^*,\textrm{ev})$.
The Grothendieck-Witt group $GW(V)$ is the Grothendick group of the 
category of symmetric
forms on $V$, modulo the relation that $(E,\theta)=h(L)$ if $E$
has a Lagrangian $L$ (a subobject such that $L=L^\perp$).  
There is a hyperbolic map $H:K_0(V)\to GW(V)$, and
the Witt group $W(V)$ is defined to be the cokernel of $H$.
When $V$ is connected, there is a canonical surjection $W(V)\to\Z/2$,
and we define $I(V)$ to be its kernel.

\smallskip
When $\dim(V)\le3$ and $V$ is smooth, $W(V)$ injects into $W(F)$, 
where $F=k(V)$ is the function field of $V$ over $k$, and there is a
well known exact sequence
\[
0\to W(V) \to W(F) \to \oplus_x W(k(x)) 
\]
where $x$ runs over all points of codimension 1.
(See \cite{Pardon} and \cite{BalmerW}, for example.)
There is a natural filtration on $W(F)$ by the powers
$I(F)^n$ of the maximal ideal $I(F)$,
and $I(F)^n/I(F)^{n+1}\cong H_\et^2(F,\Z/2)$ by \cite{OVV}.

Still assuming that $\dim(V)\le3$,
we define $I_n=I_n(V)$ to be the ideal $W(V)\cap I(F)^n$ of $W(V)$;
$I(V)=I_1(V)$.
Parimala showed in \cite{Parimala} that above sequence restricts to
an exact sequence
\begin{equation}\label{I_n-seq}
0\to I_n(V) \to I^n(F) \to \oplus_x\ I^{n-1}(k(x)) 
\end{equation}
which maps to the Bloch-Ogus sequence \cite{BlochOgus}
\begin{equation}\label{BO-seq}
0\to H^0(V,\cH^n) \to H^n_\et(F,\Z/2) \to \oplus_x\ H^{n-1}_\et(k(x),\Z/2).
\end{equation}
Here $\cH^n$ is the Zariski sheaf associated to the
presheaf $U\!\mapsto\! H^n_\et(U,\Z/2)$.

\smallskip
Recall that the discriminant of a diagonal form $(a_1,...,a_n)$
over a field $F$ is the class of $(-1)^{[n/2]}\prod a_i$ in 
$F^\times/F^{\times2}\!.$
As the hyperbolic form $H(1)$ has discriminant $+1$,
this induces a function $\disc\!: W(F)\!\to\!F^\times\!/F^{\times2}\!.$

\begin{defn}\label{def:disc}
The map $I(V)\to I(F)\to I(F)/I^2(F)\cong\! F^\times/F^{\times2}$
factors through the subgroup $H^0(V,\cH^1)$ of 
$F^\times/F^{\times2}$. The {\it algebraic discriminant}
is the induced map $I(V)\to H^0(V,\cH^1)\cong H^1_\et(V,\Z/2)$.

Similarly, the {\it Hasse invariant} 
$I(V)\to H^0(V,\cH^2)\cong {_2}\!\Br(V)$
is the homomorphism induced by the morphism
$I(V)\to I(F)$ followed by the 
Stiefel--Whitney class $w_2:I(F)\to H^2_\et(F,\Z/2)$;
see \cite[3.1]{Milnor}.
\end{defn}

\begin{example}\label{disc(theta)}
The algebraic discriminant $I(V)\to H^1_\et(V,\Z/2)$ is onto. 
To see this, recall from Kummer Theory \cite[III.4]{Milne} that
there is a split exact sequence
\[
0\to\cO^\times(V)/2 \to H^1_\et(V,\Z/2)\to{_2}\!\Pic(V) \to0.
\]
If $a$ is a global unit of $V$ and $\cO_V$ is the trivial line bundle
then $(\cO_V,a)$ is a symmetric form, representing an element of $W(V)$
whose discriminant is the class of $a$ in $\cO^\times(V)/2$. 
If $L$ is a line bundle with
$L\otimes L\cong\cO_X$, the isomorphism $\theta:L\cong L^*$ defines a 
symmetric form, and the discriminant of $(L,\theta)$ maps to the 
class of $L$ in ${_2}\!\Pic(V)$.

If $\theta$ is any symmetric form on an algebraic
vector bundle $E$, the discriminant and Hasse invariant of 
$(E,\theta)$ 
may be calculated by passing to an open 
subvariety $U$ of $V$ where $E\cong\cO_U^n$ and $\theta$ is isomorphic to the
diagonal form $(a_1,...,a_n)$, where $a_i\in H^0(U,\cO^\times)$; 
the discriminant of $(E,\theta)$ is the class of the product 
$(-1)^{[n/2]}\prod a_i$ in 
$H^0(U,\cO^\times)/2\subset F^\times/F^{\times2}$, and the Hasse invariant
of $(E,\theta)$ is the class of $\prod_{i<j}\{a_i,a_j\}$ in $H^2_\et(F,\Z/2)$.
\end{example}

Let $\tH^1_\et(V,\Z/2)$ denote the cokernel of
$H^1_\et(k,\Z/2)\to H^1_\et(V,\Z/2)$.

\begin{lemma}\label{In/In+1}
If $\dim(V)\le3$ then 
$I_n(V)/I_{n+1}(V)\to H^0(V,\cH^n)$ is an injection for all $n\ge0$.
For $n=1$, this is the discriminant isomorphism
$$
I_1/I_2 \;\map{\simeq}\; H^0(V,\cH^1)\cong H^1_\et(V,\Z/2).
$$
If $-1$ is not a square in $k(V)^\times$ then
$W(V)/I_2(V) \cong \Z/4\oplus\tH^1_\et(V,\Z/2)$.
\end{lemma}
\goodbreak

\begin{proof}
This follows from a diagram chase on the map between sequences
\eqref{I_n-seq} and \eqref{BO-seq}, using the isomorphisms 
$I^n(F)/I^{n+1}(F)\cong H^n_\et(F,\Z/2)$ 
(which are known to hold for all $n$ by \cite{OVV}).
It is also well known that $W(F)/I(F)^2$ contains $\Z/4$
if and only if $\{-1,-1\}\ne0$ in $H_\et^2(k(V),\Z/2)$
(see, e.g., \cite[3.3]{Milnor}).
\end{proof}

\begin{subremark}\label{I_2/I_3}
The injection $I_2/I_3 \to H^0(V,\cH^2)\cong {_2}\!\Br(V)$
is called the {\it Hasse invariant} \cite{Parimala}; 
it is not known if it is an isomorphism for non-affine $V$ 
over $k$, even if $k=\R$.
\end{subremark}

\goodbreak
\medskip{\it Connection to signature}\medskip

Now suppose that $V$ is a variety over $\R$,
and $V(\R)$ has $\nu>0$ connected components.
The torsion subgroup of $W(V)$ is 2-primary (Pfister \cite{Pfister}), and 
Mah\'e \cite{Mahe} and Brumfiel \cite{Brumfiel84} proved that
the signature $W(V)\to\Z^\nu$ maps the torsionfree part of $W(V)$ 
isomorphically onto a subgroup of finite index in $\Z^\nu$.
By \cite[2.4]{KSW}, the signature 
factors through $WR(V_\top)$.

Recall that the cup product with $\{-1\}\in H^0(V,\cH^1)$ induces 
stabilization maps $H^0(V,\cH^n)\to H^0(V,\cH^{n+1})$;
Colliot-Th\'el\`ene and Parimala showed \cite{CTParimala} 
that this map is an isomorphism for $n>d=\dim(V)$ and that this stable
value is $(\Z/2)^\nu$, where $\nu$ is the number of real components of $V$. 
When $n\le d$, the composite
\[
H^0(V,\cH^n)\to H^0(V,\cH^{d+1}) \cong (\Z/2)^\nu
\]
is called the {\it stabilization map}.
Example \ref{delPezzo} shows that 
$H_\et^1(V,\Z/2)\cong H^0(V,\cH^1)$ need not map onto $(\Z/2)^\nu$. 
In contrast, van Hamel proved in  \cite[2.8]{vanHamel} that the 
stabilization map $H^0(V,\cH^d)\to H^0(V,\cH^{d+1})\cong (\Z/2)^\nu$
is a surjection. For surfaces, this becomes a surjection
${_2}\!\Br(V)\cong H^0(V,\cH^2) ~\map{\textrm{onto}}~ (\Z/2)^\nu$.

\begin{example}\label{delPezzo}
A nonsingular cubic surface in $\bP^3_\C$ is a special type of 
{\it Del Pezzo} surface, and is birationally equivalent to $\bP^2_\C$
(see \cite[V.4.7.1]{Hart}). There are nonsingular cubic surfaces defined 
over $\R$ which are not birationally equivalent to $\bP^2$ over $\R$;
they have $\Pic(V)\cong\Z^3$, and the real locus $V(\R)$ has 
$\nu=2$ components, homeomorphic to $S^2$ and $\R\bP^2$ respectively; 
see \cite[V.5.4]{Silhol}.
Since $H^1_\et(V,\Z/2)\cong \Z/2$ by Kummer theory, it follows that
$H^1_\et(V,\Z/2)\to(\Z/2)^\nu$ is not onto.
\end{example}

\goodbreak
\medskip{\it The torsion subgroup}\medskip

The torsion subgroup $I(V)_\tors$ of $I(V)$ has an induced filtration
by the ideals $I_n(V)_\tors = I_n(V) \cap I(V)_\tors$.

\begin{proposition}\label{exp.2^d}
If $V$ is a smooth $d$-dimensional variety over $\R$, then
\\
(i) $I_{d+1}(V)$ is torsionfree as an abelian group; $I_{d+1}(V)_\tors=0$.
\\(ii) the torsion subgroup $I(V)_\tors$ of $I(V)$ has exponent $2^{d}$.
\\(iii) If $V(\R)=\emptyset$ then $W(V)$ is an algebra over $\Z/2^{d+1}$.
\end{proposition}

\begin{proof}
It is classical \cite[Thm.\,E]{ElmanLam} that $I^{d+1}(F)$ is a torsionfree
abelian group. Hence the subgroup $I_{d+1}(V)$ is also torsionfree.
Since the element '2' of $W(V)$ lies in $I(V)$ and 
$I^n(V)\subseteq I_n(V)$, $2^{d+1}\in 2I_d(V)\subseteq I_{d+1}(V)$.
Finally, (iii) is immediate from (ii) and the fact that the
torsionfree part has rank $\nu$.
\end{proof}

%
%
When $\dim(V)=2$, $I_3(V)$ is torsionfree and $I(V)_\tors$ has exponent~4.
\[
I(V)_\tors \supseteq I_2(V)_\tors \supseteq I_3(V)_\tors=0.
\]

Following Sujatha \cite{Sujatha},
let $H^0_\tors(V,\cH^n)$ denote the kernel of the stable map
$H^0(V,\cH^n)\to H^0(V,\cH^{d+1})\cong(\Z/2)^\nu$.
Sujatha described the torsion subgroup $I(V)_\tors$ as follows.

\begin{proposition}[Sujatha]\label{Itorsion}
When $V$ is a smooth surface over $\R$, 
there is a short exact sequence
\[
0\to H^0_\tors(V,\cH^2) \to I(V)_\tors \to H^0_\tors(V,\cH^1) \to 0.
\]
\end{proposition}

\begin{proof}
Sujatha proves in \cite[2.1, 2.2]{Sujatha} that the maps
\[
I(V)_\tors/I_2(V)_\tors\!\to\! H^0_\tors(V,\cH^1), \kern7pt 
I_2(V)_\tors/I_3(V)_\tors\!\to\! H^0_\tors(V,\cH^2)
\]
are isomorphisms, and $I_3(V)_\tors=0$ by Proposition \ref{exp.2^d}.
\end{proof}

\begin{subremark}\label{def:j,k}
Set $j=\dim H^0_\tors(V,\cH^1)$
and $k=\dim H^0_\tors(V,\cH^2)$.
It follows that (when $\nu>0$)
the torsion subgroup $W(V)_\tors$ of $W(V)$ is a group of
exponent~4 and order $2^{j+k}$. (See \cite[Lemma 3.2]{Sujatha}.)
Note that $k+\nu=\dim {_2}\!\Br(V)$ 
by van Hamel's result \cite[2.9]{vanHamel}.

Sujatha \cite[3.1]{Sujatha} also defines a group $N$ 
fitting into an exact sequence
\[
{_2}\Pic(V) \to {_2}\Pic(V_\C) \to N \to \Pic(V)/2 \to \Pic(V_\C)/2.
\]
Her invariant $l$ is the dimension of $N$; she proves in
\cite[3.3]{Sujatha} that there are $k+l$ summands in $W(V)_\tors$,
and $j-l$ summands $\Z/4$ in $W(V)$.  It follows by a counting argument
that there are $k+2l-j$ summands $\Z/2$ in $W(V)$.
However, the numbers $j,k$ and $l$ are not always easy to determine; 
see Krasnov \cite{Krasnov-Br},  \cite{Krasnov02}.
\end{subremark}

Recall that a variety $V$ over $\R$ is {\it geometrically connected}
if $V\times_{\R}\C$ is connected, i.e., if the function field of $V$
does not contain $\C$.

\begin{corollary}\label{W/I2}
Let $V\!$ be a geometrically connected surface over $\R$
with no real points. 
Then we have a short exact sequence
\[
0 \to {}_2\!\Br(V)\to W(V) \to \Z/4\times\tH^1_\et(V,\Z/2) \to 0,
\]
$\tH^1_\et(V,\Z/2)$ denoting the cokernel of
$H^1_\et(\Spec\,\R,\Z/2)\!\to\! H^1_\et(V,\Z/2)$.
\end{corollary}
\goodbreak

\begin{proof}
Since $\nu=0$, Proposition \ref{Itorsion} implies that
$I_2(V)\cong {_2}\!\Br(V)$. We conclude by Lemma \ref{In/In+1}.
\end{proof}

The group extension in Corollary \ref{W/I2} need not split,
as illustrated by Example \ref{Motzkin}. 

\begin{subremark}
Corollary \ref{W/I2} recovers Sujatha's result \cite[3.4]{Sujatha} that 
if $V$ is a geometrically connected surface with $V(\R)\!=\!\emptyset$ then
$W(V)$ has order $2^{j+k+1}$, because
$j=\dim H^1_\et(V,\Z/2)$ and $k = \dim {_2}\!\Br(V)$. 
\end{subremark}

\medskip
Let $Q_1$ denote the projective Brauer-Severi curve $x^2+y^2+z^2=0$.

\begin{lemma}\label{exp.8}
Suppose that $V$ is a smooth geometrically connected surface
with no real points. Then $W(V)$ is a $\Z/8$-algebra 
and surjects onto $\Z/4$. It is a $\Z/4$-algebra if and only if
there is a morphism $V\to Q_1$ defined over $\R$.
\end{lemma}

\begin{proof}
Since $W(V)$ injects into $W(F)$, it suffices to study
the exponent of $W(F)$.  When $V$ is geometrically connected 
($\R(V)$ does not contain $\C$), 
then $-1$ is not a square in $F$. 
Hence the element '2' of $W(F)$ is nonzero, because its
discriminant is $-1$ in $F^\times/F^{\times2}$.
The group $W(F)$ has $I^3(F)=0$ by Proposition \ref{exp.2^d}(iii)
(or \cite[Thm.\,E]{ElmanLam}).
Thus the ideal $I(F)$ has exponent~4, and '8' is zero in $W(F)$.
Thus the image of $\Z$ in $W(F)$ is either $\Z/4$ or $\Z/8$.

It is well known \cite[3.3]{Milnor} that 
the element '4' in $W(F)$ corresponds to the
symbol $\{-1,-1\}$ in $I^2(F)\cong H^2_\et(F,\Z/2)$.
This symbol vanishes if and only if there is a morphism
$V\to Q_1$ defined over $\R$.
\end{proof}

\begin{subremark}
Milnor showed in \cite[3.1, 3.3, 4.1]{Milnor} that $1+w_1+w_2$ 
is an additive isomorphism from $I(F)/I^3(F)$ to
$1+H_\et^1(F)+H_\et^2(F)$, which is a group under
$(1+a_1+a_2)(1+b_1+b_2)=1+(a_1+b_1)+(a_2+b_2+a_1\cup b_1)$,
with the inverse given by $s_1$ and $s_2$. 
For example, the element '2' of $I(F)$ 
maps to $1+\{-1\}$ and '4' maps to $(1+\{-1\})^2=1+\{-1,-1\}$.
\end{subremark}
\goodbreak

\begin{example}\label{Motzkin}
There are real surfaces $V$ for which the image of $\Z\to W(V)$ is $\Z/8$;
since $W(V)\subset W(F)$, this 
is equivalent to the condition that
the function field $F$ of $V$ has level 4 
(i.e., $-1$ is a sum of 4 squares).
As pointed out by Parimala in \cite{ParimalaSujatha}, 
this is the case if $V$ is the hypersurface
in $\bP^3_{\R}$ defined by $z^2-f$, where $f(x,y)$ is 
the Motzkin polynomial $x^4y^2+x^2y^4-3x^2y^1+1$
(or certain other positive polynomials).
\end{example}

\begin{example}\label{W(ExE)}
Consider the real surface $V=E\times E$, where $E$ is 
an elliptic curve with 2 real connected components; 
this is a real form of the abelian variety $V_\C=E_\C\times E_\C$. 
To compute $W(V)$ using Remark \ref{def:j,k}, 
we first compute $H^*_\et(V,\Z/2)$, using the Leray spectral sequence 
\[ H^p(G,H^q_\et(V_\C,\Z/2))\Rightarrow H^{p+q}_\et(V,\Z/2). \]
Because $E(\R)$ contains the exponent~2 points of the group
$E(\C)$, $G$ acts trivially on $H^*_\et(E_\C,\Z/2)$.
Dropping the coefficients $\Z/2$ from the notation,
the K\"unneth formula shows that $G$ also acts trivially on 
$H_\et^1(V_\C)\cong(\Z/2)^4$ and
$H_\et^2(V_\C)\cong(\Z/2)^6$.
We get an exact sequence
\[ 0\!\to (\Z/2)\to H^1_\et(V) \to H^1(V_\C)^G \to 0. \]
Hence $H^1_\et(V)\cong (\Z/2)^5.$ In the corresponding
spectral sequence for $H^*_\et(E)$, the differential
$\Z/2=E_2^{0,2}(E)\to E_2^{2,1}(E)=(\Z/2)^2$ is an injection. 
Hence the differential from $E_2^{0,2}(V)=H^2_\et(V_\C)$ to 
$E_2^{2,1}(V)=H^1_\et(V_\C)$
has rank 2, and we deduce that $H^2_\et(V)\cong (\Z/2)^7.$

Since the stabilization map $H_\et^1(V)\to H_\et^2(V)$ is the cup product
with $[-1]$, we also see that $H^1_\tors(V)\cong(\Z/2)^3$ 
injects into $H^2_\tors(V)\cong(\Z/2)^4$, so Sujatha's invariant $l$
vanishes, and $j=3$, $k=4$.  It follows from 
Remark \ref{def:j,k} that 
\[ W(E\times E)\cong \Z^4 \oplus(\Z/4)^3\oplus \Z/2. \]
\end{example}

\goodbreak
\section{$WR$ and the signature}\label{sec:WR} 

In this section, we define the signature map on $WR(X)$ and relate it 
to the edge map $\eta^0$ in the Bredon spectral sequence for $KO_G(X)$.
First we recall the definition of $WR(X)$ from \cite{KSW}.

Let $X$ be a topological space with involution.
By a {\it Real vector bundle} on $X$ we mean a complex vector bundle
$E$ with an involution $\sigma$ compatible with the involution on $X$
and such that for each $x\in X$ the isomorphism
$\sigma:E_x\to E_{\sigma x}$ is $\C$-antilinear. 
Following Atiyah \cite{Atiyah}, we write $KR(X)$
for the Grothendieck group of Real vector bundles on a compact space $X$.

Recall from \cite{KSW} that the Grothendieck-Witt group $GR(X)$ is
the Grothendieck group of the category of symmetric forms $(E,\phi)$,
where $E$ is a Real vector bundle
on $X$ and $\phi$ is an isomorphism of Real vector bundles
from $E$ to its dual $E^*$ such that $\phi=\phi^*$.

\begin{defn}\label{def:WR}
The Real Witt group $WR(X)$ of a compact $G$-space $X$
is the cokernel of the hyperbolic map $H:KR(X)\smap{} GR(X)$.

By \cite[Thm.\,2.2]{KSW}, there is an isomorphism $KO_G(X)\cong GR(X)$
identifying the hyperbolic map with the map $KR(X)\to KO_G(X)$,
sending a Real vector bundle to its underlying $\R$-linear bundle
with involution.
Thus $WR(X)$ is also the cokernel of $KR(X)\to KO_G(X)$.
\end{defn}

\begin{subex}\label{trivial action}
It is easy to see that $KO_G(\pt)\cong RO(G)$ and the signature
$WR(\pt)\cong\Z$ sends $[\R]$ and $[\R(1)]$ to $1$ and $-1$.
Suppose more generally that $G$ acts trivially on $X$. Then
$KO_G(X)\!\cong KO(X)\otimes R(G)$; we showed in
\cite[2.4]{KSW} that $KR(X)\cong KO(X)$ and $WR(X)\cong KO(X)$.
In this case, the rank of a $KO$-bundle on each component of $X$ 
defines a map $WR(X)\to\Z^\nu$, where $\nu=|\pi_0(X)|$.
\end{subex}

In contrast, $WR(X)$ is a $\Z/2$-algebra when $X=G\times Y$. 
In this case, $WR(X)$
is the cokernel of the natural map $KU(Y)\to KO(Y)$ and 
hence a $\Z/2$-algebra; see \cite[2.4(b) and 2.6]{KSW}.
The following example illustrates this.

\begin{lemma}\label{Arason}
When $V$ is complex projective space $\bP^d_\C$, 
$X=(\bP^d_\C)_\top$ is $G\times\C\bP^d$ and
\[
W(\pt) \cong W(\bP^d_\C) \cong WR(X) \cong \Z/2.
\]
\end{lemma}

\begin{proof}
The isomorphism $\Z/2=W(\pt)\cong W(\bP^d_\C)$ 
is due to Arason \cite{Arason}.
As noted above, 
$WR(X)$ is the cokernel of $KU(\C\bP^d)\to KO(\C\bP^d)$. In this case, 
Karoubi and Mudrinski \cite{KM} proved that the forgetful map
$KU(\C\bP^d,\pt)\to KO(\C\bP^d,\pt)$ is onto, so 
$WR(\C\bP^d)=\Z/2$ as well.
\end{proof}

\paragraph{\bf The topological signature.}
Now let $X$ be any compact $G$-space, such that $X^G$ has $\nu$ components.
The signature of a symmetric form $(E,\phi)$ on $X$ is an
element of $\Z^\nu$; the signature on a component of $X^G$ is
the signature of the form $\phi_x$ on the vector space $E_x$, 
for any point $x$ of the component (it is independent of the 
choice of $x$). This passes to a homomorphism $KO_G(X)\to\Z^\nu$.

Since hyperbolic forms have signature 0, it also defines
a topological {\it signature} $\sigma:WR(X)\smap{}\Z^\nu$;
the signature of $w\in WR(X)$ is the signature of a quadratic form 
$(E,\phi)$ representing $w$. While $\sigma$ may not be onto
(see Remark \ref{sign-mod2}),
its image has finite 2-primary index in $\Z^\nu$ by \cite[7.4]{KSW}.

As observed in \cite{KSW}, there is a canonical map $W(V)\to WR(V_\top)$
and the composition with the topological signature 
is the classical algebraic signature.


On each component, the signature of an $n$-dimensional symmetric form 
$\phi$ is congruent to $n$ modulo~2.  
If $X$ is connected, it follows that the signature
of $\phi$ lies in the subgroup $\Gamma$ of $\Z^\nu$, 
defined as follows.

\begin{defn}\label{def:Gamma}
Let $\Gamma$ denote the subgroup of $\Z^\nu$
consisting of all $(c_1,...,c_\nu)$ with either 
all $c_i$ even or all $c_i$ odd; $\Gamma$ contains $2\Z^\nu$ as
an index~$2$ subgroup.
The group $\Gamma$ is free abelian of rank $\nu$, but there is no
isomorphism of $\Gamma$ with $\Z^\nu$ which is natural in $X$.
\end{defn}

We now assume $X$ is connected, and
relate the signature to the canonical edge map
\begin{equation}\label{eq:rank}
\eta^0: KO_G(X) \to H^0_G(X;KO_G^0) \cong \Z\oplus\Z^\nu
\end{equation}
in the Bredon spectral sequence \eqref{Bredon-ss}; see Lemma \ref{H^q_G}.
The first coordinate of $\eta^0$ is the rank of the underlying vector bundle.
To describe the other coordinates of $\eta^0$, 
let $\widetilde{KO}_G(X)$ denote the kernel of the rank map 
$KO_G(X)\!\to\!\Z$. For each point $x$ of $X^G$\!, 
$\widetilde{KO}_G(x)$ is isomorphic to $\Z$ on $[\R]\!-\![\R(1)]$.
This yields a map 
$\widetilde{KO}_G(X)\to \widetilde{KO}_G(x)\cong\Z$; 
if $E$ is a $G$-vector bundle of rank $r$ on $X$, the map sends
$[E]-r$ to $[E_x]-r$.  Choosing a point on each of the $\nu$
components of $X^G$ yields the edge map $\eta^0$ 
from $\widetilde{KO}_G(X)$ to $\Z^\nu$. It sends the line bundles
$X\!\times\!\R$ and $X\!\times\!\R(1)$ to $(0,...,0)$
and $(-1,...,-1)$.

Both $H^0_G(X;KO_G)$ and $\nu$ 
are determined by the 1--skeleton of $X$.

\goodbreak
\begin{lemma}\label{signature}
Let $X$ be a connected finite $G$-CW complex whose fixed locus $X^G$\!
has $\nu$ components.  The signature 
$WR(X)\!\smap{\sigma}\!\Z^\nu$\! 
is induced by
\[
KO_G(X)\map{\eta^0} H^0_G(X;KO_G) 
\cong\Z\oplus\Z^\nu\ \map{\rho}\ \Gamma \subseteq \Z^\nu\!,
\]
where $\eta^0$ is the edge map in \eqref{Bredon-ss}
and $\rho$ is given by 
$$\rho(r,a_1,...,a_\nu)=(r+2a_1,...,r+2a_\nu).$$
\end{lemma}
\goodbreak

\begin{proof}
If $\pt$ is a single point, we choose $[\R]$ and
$\lambda=[\R]-[\R(1)]$ as a basis of 
$H^0_G(\pt,\Z)\cong RO(G)\cong\Z^2$.
The projection `$+$' in Example \ref{ex:coeff}(a)
from $H^0_G(\pt,\Z)\cong RO(G)$ to $H^0(\pt/G,\Z)\cong\Z$
is split by the map sending the generator to $[\R]$, and
the kernel is generated by $\lambda$.

The image of $KR(\pt)\cong\Z$ in $H^0_G(\pt,\Z)$ is generated by 
$[\R]+[\R(1)]=2[\R]-\lambda$; it is the kernel of the map
$H^0_G(\pt,\Z)\map{s}\Z$ sending $(r[\R]+a\lambda)$ to its signature $r+2a$.
Therefore $s$ induces the signature $\sigma:WR(\pt)\map{\cong}\Z$.

Now let $Y$ denote $\nu$ discrete points. It follows that 
the signature $WR(Y)\smap{\cong}\Z^\nu$ is induced by the 
composition of the map
\[
KO_G(Y)\map{\cong} H^0_G(Y,\Z)\map{\cong} 
H^0(Y^G,\Z)\oplus H^0(Y/G,\Z) \cong \Z^\nu\oplus \Z^\nu,
\]
with the map $\Z^{2\nu}\smap{s}\Z^\nu$,
sending $(r_1,....,a_1,...)$ to $(r_1+2a_1,...,r_\nu+2a_\nu)$.  
The general case now follows from the following diagram,
whose rows are from Lemma \ref{H^q_G}
\begin{equation*}
\xymatrix@R=1.5em{
0\ar[r]&H^0(X^G,\Z)\ar[d]^{\cong}\ar[r]&
H^0_G(X;KO_G)\ar[r]^{+}\ar[d]^{\mathrm{into}}&
H^0(X/G,\Z) \ar[r]\ar[d]^{\mathrm{diag}}&0 \\
0\ar[r]&H^0(Y^G,\Z)\ar[r]^{\lambda}&
H^0_G(Y;KO_G)\ar[r]^{+}&H^0(Y,\Z) \ar[r]&0
}
\end{equation*}
and whose columns are
associated to the evident map $Y\to X^G\subseteq X$, obtained by
choosing a basepoint for each component of $X^G$.
By inspection, the map $\rho:H^0_G(X;KO_G)\to\Gamma$ is the
restriction of the map $s:H^0_G(Y;KO_G)\to\Z^\nu$. It follows that
$\rho\eta^0_X$ is the restriction of $s\eta^0_Y$, which is
the signature.
\end{proof}

\begin{subremark}\label{H(1)}
The sequence $KR(X)\smap{}\Z^{1+\nu}\smap{\rho}\Gamma$
of Lemma \ref{signature} is exact.
Indeed, the kernel of $\rho$ is the subgroup of $\Z^{1+\nu}$
generated by $(2,-1,...,-1)$. This is the image of
$\eta^0\circ H: KR(X)\to\Z\oplus \Z^{\nu}$, because 
$\eta^0\circ H$ maps the trivial Real line bundle $X\times\C$
to $(2,-1,...,-1)$.
\end{subremark}

\goodbreak
\section{$WR$ for low dimensional $X$}\label{sec:lowdim}

In this section, we determine $WR(X)$ when $\dim(X)\le2$.
There is no harm in assuming that $X$ is an irreducible $G$-space, 
so that $X$ is either connected or has the form $X=G\times Y$
with $Y$ connected.
We saw in \cite[7.1]{RGB} that $WR(G\times Y)\cong\Z/2\oplus H^1(Y,\Z/2)$
when $\dim(Y)\le2$.
Thus we are reduced to the case when $X$ is connected.

In general, there is a map $w_1:WR(X)\to\H^1_G(X,\Z/2)$,
defined in \ref{def:w1} and \ref{def:w_n} below.
Let us write $[-1]$ for the element $w_1(X\!\times\!\R(1))$ 
of $\H^1_G(X,\Z/2)$.
We will see in Remark \ref{Xcovers} that $[-1]$ is always nonzero
in  $\H^1_G(X,\Z/2)$. Let $\tH^1(X/G,\Z/2)$ denote the
quotient of $H^1(X/G,\Z/2)$ by the subgroup generated by 
the element $[-1]=w_1(X\!\times\!\R(1))$.

\begin{lemma}\label{WR(graph)}
When $X$ is a connected 1-dimensional $G$-complex,
and $X^G$ has $\nu$ components, we have 
\[  WR(X) = 
\begin{cases}\Z^\nu \oplus H^1(X/G,\Z/2), & \nu>0, \\ 
             \Z/4\oplus \tH^1(X/G,\Z/2), &\nu=0.
\end{cases}\]
\end{lemma}

\begin{proof}
The spectral sequences of  \eqref{Bredon-ss} for $KO_G$ and $KR$
converge, having only two nonzero rows.  By Lemma \ref{H^q_G}, we have
\[
  KO_G(X)\cong\Z\oplus\Z^\nu\oplus H^1_G(X;KO_G^{-1}),\qquad
  KR(X)\cong\Z\oplus H_G^1(X;KR^{-1}),
\]
 and  the cokernel of $H^1_G(X;KR^{-1})\to\! H^1_G(X;KO_G^{-1})$
is $H^1(X/G,\Z/2)$. Thus it remains to consider the cokernel
of the map of $H^0_G$-components $\Z\to\Z\oplus\Z^\nu$.
It is described in Remark \ref{H(1)}, and
yields the displayed calculation.
\end{proof}

\begin{example}\label{WR(T)}
There are three actions of $G$ on the circle. Recall that
$S^{p,q}$ denotes the sphere in $\R(1)^p\times\R^q$.

(i) $S^{0,2}$ is the circle with trivial $G$-action. Let $\mu$
denote the M\"obius line bundle. Then
$KO_G(S^{0,2})$ is $\Z^2\oplus(\Z/2)^2$ with generators the line bundles
$1$, $X\times\R(1)$, $\mu$ and $\mu\otimes \R(1)$, while 
$WR(S^{0,2})\cong KO(S^{0,2})$ is $\Z\oplus\Z/2$ 
by Example \ref{trivial action}.

(ii) $S^{2,0}$ is the circle with the antipodal involution on $S^1$. 
Lemma \ref{WR(graph)} yields $WR(S^{2,0})\cong\Z/4$.

(iii) $S^{1,1}$ is the unit circle in $\C$, with the induced complex
conjugation. By Lemma \ref{WR(graph)},
\[  
WR(S^{1,1}) \cong \Z^2.
\]
The Real line bundle $E=S^{1,1}\times\C$ carries a canonical
Real symmetric form $\theta=t$, which is multiplication by $t$ on 
the fiber over $t\in S^{1,1}$.  By \cite[Thm.\,2.2]{KSW}, this corresponds
to the $G$-subbundle of $E$ whose fiber over $t$ is $it\cdot\R$.
It is a nontrivial element of $\Pic_G(S^{1,1})\cong\{\pm1\}^2$,
and its discriminant 
$\iota=w_1(E,t)$ is nontrivial.

If $X\smap{a}S^{1,1}$ is an equivariant map, the pullback of $(E,t)$ 
under the map $a^*:\Pic_G(S^{1,1})\to \Pic_G(X)$ is the class of 
$(X\!\times\!\C,a)$ in $\Pic_G(X)$, and 
its discriminant is $w_1(a^*(E,t))=a^*(\iota)$.
\end{example}

\begin{lemma}\label{2-mfld}
Let $X$ be a closed connected 2-manifold with involution,
such that $X^G$ is a union of $\nu$ circles.  Then
\[
WR(X) = \begin{cases}\Z^\nu \oplus (\Z/2)^h, &\nu>0 \\
\Z/4\times (\Z/2)^{h-1}, &\nu=0,\end{cases}
~ \textrm{where~} h=\dim H^1(X/G,\Z/2).
\]
\end{lemma}

\goodbreak
\begin{proof}
If $\nu=0$, this was proven in \cite[8.1]{RGB}, so suppose $\nu>0$.
Since $X/G$ is a 2-manifold with boundary,  $H^2(X/G,\Z/2)=0$.
(See the argument of \cite[1.9]{PW}.) As $H^2(X^G,\Z/2)=0$,
Lemma \ref{H^q_G} yields
$H^2_G(X;KO_G^{-2})=0$ and $H^1_G(X;KO_G^{-1})=(\Z/2)^{h+\nu}$.  
By \eqref{Bredon-ss}, we have
$KO_G(X)=\Z^{\nu+1}\oplus (\Z/2)^{h+\nu}$. 
Now use the argument of Lemma \ref{WR(graph)}.
\end{proof}

If $C$ is a smooth projective curve over $\R$, $C_\top$ is 
a 2-manifold whose genus $g$ is $h$ when $\nu>0$, and $h-1$ when $\nu=0$.
Thus we recover the following calculation of $WR(C_\top)$, 
which was derived in \cite[Thms.\,4.6 and 4.7]{KSW}, using
$H^2_\et(C,\Z/2)\cong(\Z/2)^{2\nu}$ and $\Br(C)\cong(\Z/2)^\nu$.
(The fact that $W\cong WR$ for curves is also in {\it loc.\,cit.})


\goodbreak
\begin{corollary}\label{WR(curve)}
If $C$ is a smooth geometrically connected projective curve of genus $g$
over $\R$, with $\nu$ real components, then
\[
W(C)\cong WR(C_\top) \cong \begin{cases}
\Z^\nu\oplus (\Z/2)^g, & \nu>0, \\ \Z/4\oplus (\Z/2)^g, &\nu=0.
\end{cases}\]
\end{corollary}

\begin{subremark}\label{affine curve}
When $C(\R) = \emptyset$, and $U\subset C$ is the complement of
$n>0$ (complex) points, then $W(U)\cong W(C)\oplus(\Z/2)^{n-1}$.
This follows from Corollary \ref{WR(curve)}, \eqref{I_n-seq},
and the fact that $\Pic(C)/2\cong\Z/2$. 
In addition, $W(U)\cong WR(U)$ by \cite[4.1]{KSW}.
We leave the details to the reader. 
\end{subremark}

\begin{example}\label{WR(2-sphere)}
We can now calculate $WR(X)$ of a 
2-sphere with involution.
Writing $S^{p,q}$ for the sphere in $\R(1)^p\times\R^q$, we have:
\[
WR(S^{0,3})\cong\Z\oplus\Z/2;\quad 
WR(S^{1,2})\cong\Z;\quad WR(S^{2,1})\cong\Z^2\oplus\Z/2;
\]
and $WR(S^{3,0})\cong\Z/4$. 
The calculation $WR(S^{0,3})\cong KO(S^2)\cong\Z\oplus\Z/2$ 
is from Example \ref{trivial action}.
The formulas for the genus~0 curves $S^{1,2}=\C\bP^1$ and $S^{3,0}=Q_1$
are from Lemma \ref{2-mfld}. 

For $X=S^{2,1}$, we note that $X^G$ is 2 points and $X/G\simeq S^2$.
Therefore $H_G^1(S^{2,1},\Z/2)=0$ and $H_G^2(X;KO^{-2})\cong\Z/2$, 
using Lemma \ref{H^q_G}. Thus $KO_G(X)\cong\Z^3\oplus \Z/2$.
We also have $\H_G^1(S^{2,1}\!,\Z/2)\cong\Z/2$.
As in the proof of Lemma \ref{2-mfld}, we have $KR(X)=\Z$, and
hence $WR(S^{2,1}) 
\cong\Z^2\oplus \Z/2$.

More generally, if $q>0$ then
$KR^n(S^{p,q}\!,\pt) \cong KO^{n+p-q+1}(\pt)$.
This follows easily from Atiyah's sequence \cite[(2.2)]{Atiyah}
for $(X,Y)=(B^{p,q},S^{p,q})$.
In addition, because $S^{p,q}$ is the 1--point compactification of
$\R(1)^p\times\R^{q-1}$,
$KO_G(S^{p,q},\pt)$ is isomorphic to the $K$-theory with compact supports
$KO_G^{1-q}(\R(1)^p)$, a group isomorphic to the topological
$K$-theory $K_\top^{1-q}(C^{0,p+1})$ of the Clifford algebra $C^{0,p+1}$;
see \cite{MK-8}.

If $G$ acts trivially on $Y$ then
$KO_G(Y\times S^{p,q},Y)\cong K_\top^{1-q}(C^{0,p+1}(Y))$,
where $C^{0,p+1}(Y)$ is the Banach algebra of continuous functions from
$Y$ to $C^{0,p+1}$.

For completeness, we note that when $q=0$ and $p>0$ the group
$KO_G(Y\times S^{p,0})$ is $KO(Y\times\R\bP^{p-1})$. This group
was computed by Adams (when $Y$ is a point) and in general by 
the first author in \cite{MK-8}.

For later use in Example \ref{ExE}, we record that
\[
KO_G(S^{1,3}\!,\pt)\cong\Z/2,~ KO_G(S^{2,2}\!,\pt)=0 
\text{ ~and~ } KO_G(S^{2,3}\!,\pt)=\Z.
\]
Using this, we can check that
$WR(S^{1,3}) \cong WR(S^{2,2}) \cong WR(S^{2,3}) \cong \Z$.
(An alternative calculation of $WR(S^{2,2})$ was given in
Example \ref{S22}.)
\end{example}

%

\begin{example}\label{SGm}
Let $T$ be the torus $S^1\times S^{1,1}$, where $G$ acts trivially on $S^1$.
Then $T$ is a closed connected 2-manifold with involution,
$T^G$ is a union of $2$ circles, and $T/G\cong S^1\times[0,1]$.
From Lemma \ref{2-mfld} and Example \ref{WR(T)}(i), 
we see that $WR(S^1)\cong \Z\oplus\Z/2$ and
\[ 
WR(T)\cong \Z^2\oplus\Z/2 \not\cong WR(S^1) \oplus WR(S^1).
\]
If $C$ is the affine circle $x^2+y^2=1$ over $\R$
and $V=C\times\Gm$, then $W(C)\cong\Z\oplus\Z/2$ and
$W(V)\cong \Z^2\oplus(\Z/2)^2$ by the Fundamental Theorem for $W(C)$
\cite[pp.\,138--9]{MKlocalisation}.
Since $C_\top$ and $V_\top$ are equivariantly homotopic to 
$S^1$ and $T$, respectively, it follows that
$W(V)\not\cong WR(V_\top)$. 
\end{example}


\begin{example}\label{plane}
The projective plane $\bP^2$ over $\R$ has $\bP^2_\top=\C\bP^2$ and
\[  W(\bP^2)\!\cong WR(\C\bP^2)\!=\Z.  \]
Because $W$ and $WR$ are birational invariants of a surface (see Theorem
\ref{birational}), this also implies that 
$W(\bP^1\times\bP^1)\cong WR(\C\bP^1\!\times\!\C\bP^1)\cong \Z$.

Arason proved that $W(\bP^d)=\Z$ for all $d$;
see \cite{Arason}. 
The $G$-space $X=\bP^2_\top$ is $\C\bP^2$ with $X^G=\R\bP^2$.
The quotient $X/G$ has $H^*(X/G,\Z)\cong H^*(S^4,\Z)$. 
In fact, $X/G$ is homotopy equivalent to $S^4$; 
writing $\C\bP^2$ as the cone of the Hopf map, one sees
that $X/G\simeq\Sigma(S^3/G)\simeq S^4$.
By Lemma \ref{H^q_G}, 
\[ E_2^{p.q}=H^p_G(X;KO_G^q) =
H^p(\R\bP^2\!,KO^q)\oplus H^p(X/G,KO^q),
\] 
and the second factor vanishes unless $p=0,4$. Thus 
\[
H^p_G(X;KO_G^{-p}) = \begin{cases}  \Z^2,& p=0; \\
\Z/2,&  p=1,2;\\ 
\Z, &   p=4;\\
0,& \text{otherwise.}\end{cases}
\]
Since $KO^{n}_G(X)$ maps onto $KO^{n}_G(\pt)$ and $KO_G(X)$ maps to
$KO(\R\bP^2)=\Z\oplus\Z/4$, the spectral sequence \eqref{Bredon-ss} 
degenerates to show that $KO_G(X)$ is the extension $\Z^3\oplus\Z/4$
of $\Z^3\oplus\Z/2$ by $\Z/2$.  By Lemma \ref{H^q_G},
Remark \ref{KR(-1)}, Lemma \ref{H4-iso} and Remark \ref{H(1)},
$KR(X)$ maps onto the torsion, the final $\Z$,
and the subgroup of $\Z^2\oplus\Z/2\oplus\Z$ generated by
$(2,-1,[-1],0)$.  It follows that $WR(X)\cong\Z$.
\end{example}

\begin{subremark}\label{KR(P2)}
Atiyah proved in \cite[2.4]{Atiyah} that $KR(\C\bP^2)\cong \Z[t]/t^3$.
In the filtration arising from \eqref{Bredon-ss}, $t$ is detected by
$H^1_G(X,KR^{-1})\cong\Z/2$; $t^2$ and $2t$ are in filtration level~2.
\end{subremark}

\goodbreak\medskip 
\section{The equivariant Picard group}\label{sec:Pic_G}

In this section $X$ will be a connected $G$-space,
and $\nu$ will denote the number of components of $X^G\!.$

\begin{definition}\label{def:PicG}
Let $\Pic_G(X)$ denote the group of rank~1 real $G$-vector bundles 
on $X$ with product $\otimes$;
its unit `$1$' is the trivial line bundle $X\times\R$.
If $L$ is such a bundle,  $L\otimes L$ is trivial, and
the involution on the fibers of $L$ over fixed points is 
multiplication by $\pm1$. This involution is locally constant,
so if $X^G$ has $\nu$ components 
it defines a natural map $\Pic_G(X)\to\{\pm1\}^\nu$, 
called the {\it sign map}.

Let $\Pic^0_G(X)$ denote the kernel of the sign map.
There is a natural isomorphism $\Pic^0_G(X)\smap{\cong} H^1(X/G,\Z/2)$,
defined as follows.
If $L$ is a $G$-line bundle on $X$ with trivial sign (so 
$G$ acts trivially on $L_x$ for all $x\in X^G$, then the 
identifications $L_x\cong L_{\sigma x}$ imply that $L$
descends to a line bundle $L/G$ on $X/G$. Since real line 
bundles on $X/G$ are classified by the group $H^1(X/G,\Z/2)$, this
defines a map $\Pic_G^0(X)\to H^1(X/G,\Z/2)$. 
Conversely, the pullback of such a
line bundle along $X\to X/G$ is an equivariant line bundle on $X$
with trivial sign, i.e., an element of $\Pic_G^0(X)$.

A natural isomorphism $\Pic_G(X)\smap{\omega}\H_G^1(X,\Z/2)$,
was constructed in \cite[3.2]{RGB}.
It is compatible with the exact sequence
\[
0 \to H^1(X/G,\Z/2) \to \H^1_G(X,\Z/2)\ \map{\textrm{sign}}\ H^0(X^G,\Z/2).
\]
of Example \ref{Borel}, as we see from the map 
$\H^1_G(X,\Z/2)\to\H^1_G(X^G,\Z/2)$.
This exact sequence, 
combined with either Example \ref{delPezzo} or Example \ref{PicG=0}, 
shows that the sign map is not always onto when $\nu\ge2$.
\end{definition}

\begin{subremark}\label{Xcovers}
The $G$-line bundle $X\times\R(1)$ is nontrivial. 
This is clear if $X^G\ne\emptyset$, as $\R(1)$ has a nontrivial sign. 
When $X^G=\emptyset$, it is a nonzero element of 
$\Pic^0_G(X)\cong H^1(X/G,\Z/2)$. This is because
$X\to X/G$ is a nontrivial covering space, covering spaces
with group $G$ are classified by elements of $H^1(X/G,\Z/2)$, and
we showed in Example 2.7 of \cite{KSW} that $\omega(X\times\R(1))$ 
is the element classifying this particular cover.
\end{subremark}

\begin{subremark}
If $V$ is a geometrically ireducible variety over $\R$, and $X=V_\top$,
Cox's Theorem \cite{Cox} states that
\[ \Pic_G(X) \cong \H^1_G(X,\Z/2) \cong H^1_\et(V,\Z/2) 
\cong \cO^\times(V)/\cO^{\times2}(V)\oplus {}_2\Pic(V).
\]
If $V$ is projective then $\cO^\times(V)=\R^\times$ and
$\Pic_G(X)\cong\Z/2\oplus {}_2\Pic(V)$.
\end{subremark}

\begin{defn}\label{def:w1}
There is a determinant map $KO_G(X)\smap{\det} \Pic_G(X)$
satisfying $\det(E\oplus F)=\det(E)\otimes\det(F)$;
$\det(L)=[L]$ if $\rank(L)=1$.
Composing $\det$ with the isomorphism of Definition \ref{def:PicG}
defines a surjection $w_1: KO_G(X)\to \H^1_G(X,\Z/2)$,
which we call the {\it discriminant}; we will
relate the discriminant to the Stiefel--Whitney class
$w_1$ in Section \ref{sec:SW} below.
\end{defn}

As in the algebraic setting, the discriminant does not factor 
through $WR(X)$, because the discriminant of $H(1)$ is $-1$.
Instead, the discriminant factors through the subgroup $I(X)$
of forms in $WR(X)$ of even degree, as we shall see in
Theorem \ref{w1:factors} below.

\begin{remark}\label{sign-mod2}
If $E$ is a rank $r$ $G$-bundle on $X$, then $e=[E]-r$ 
has rank~0 in $KO_G(X)$; if $E$ has signature $a_i$ on the
$i^{th}$ component of $X^G$ then the edge map $\eta^0$ of \eqref{eq:rank}
sends $e$ to $\eta^0(e)=(0,a_1-r,...,a_\nu-r)$, which is
an element of $0\times(2\Z)^\nu$\!.
Let $\widetilde{KO}_G(X)$ denote the
kernel of $KO_G(X) \map{\mathrm{rank}} \Z$;
since $\det(e)\!=\!\det(E)$ restricts to a line bundle with sign 
$(-1)^{(a_i-r)/2}$ on the $i^{th}$ component of $X^G\!,$
we have a commutative diagram
and a map from the kernel of $KO_G(X)\map{\eta^0}\Z^{1+\nu}$
onto $\Pic^0_G(X)$. 
\begin{equation*}\xymatrix@R=1.5em{
\ker(\eta^0) \ar[r]\ar[d]^{\textrm{onto}}
 & \widetilde{KO}_G(X) \ar[r]^-{\eta^0} \ar[d]^{\det}
                    & (2\Z)^\nu \ar[d]^{\textrm{mod}\,2}\\
\Pic_G^0(X)\ar[r] &\Pic_G(X) \ar[r]^-{\textrm{sign}} & \{\pm1\}^\nu.
}\end{equation*}
As noted in Example \ref{PicG=0}, the sign map 
$\Pic_G(X)\to\{\pm1\}^\nu$ need not be onto;
it follows that the signature $KO_G(X)\to\Gamma\subseteq\Z^\nu$ 
of \ref{signature} need not be onto either.
\end{remark}

\begin{example}\label{PicG=0}
The $2$-sphere $X=S^{2,1}$ has $\nu=2$ as $X^G$ is two points,
and we saw in Example \ref{WR(2-sphere)} that 
$KO_G(S^{2,1})\cong\Z^2\oplus(\Z/2)$. It follows that
$\Pic_G(X)\cong \H_G^1(X,\Z/2)=\Z/2$. From Remark \ref{sign-mod2},
we see that the signature $KO_G(X)\to\Gamma\subseteq\Z^2$
is not onto.
\end{example}

Given an algebraic surface $V$ defined over $\R$, such that
$V(\R)$ has $\nu$ real components,
let $S_a$ and $S_t$ denote the images of the algebraic signature
$W(V)\!\to\!\Z^\nu$ and the topological signature $WR(V_\top)\to\Z^\nu$,
respectively.  Since the algebraic signature factors through 
the topological one, $S_a\subseteq S_t$.
We do not know any examples where $S_a\ne S_t$.
 
As we noted in Lemma \ref{signature},
$S_t\subseteq\Gamma$. 
Colliot-Th\'el\`ene and Sansuc proved in \cite[4.1]{CTSansuc} 
that $I_3(V)$ maps onto $8\Z^\nu$, i.e, that $S_a$ contains $8\Z^\nu$. 
(See also \cite[proof of Thm.\,2.3.2]{CTParimala}.)
Therefore we have
\[  8\Z^\nu \subseteq S_a\subseteq S_t\subseteq\Gamma\subseteq\Z^\nu. \]
We also know from \cite[8.7]{KSW} that $2S_t\subseteq S_a$.


\begin{proposition}\label{same-sign}
The images of the algebraic and topological signatures 
in $\Z^\nu/4\Z^\nu$ are the same.
\end{proposition}

\begin{proof}
Set $X=V_\top$; we may suppose $X$ connected. Since both
$W(V)/I(V)$ and $WR(X)/I(X)$
map isomorphically to $\Gamma/2\Z^\nu\cong\Z/2$, the
signature maps $I(V)$ and $I(X)$ to $2\Z^\nu$.
We need to consider their images modulo $4\Z^\nu$.


Let us write $\{\pm1\}^\nu$ for $(2\Z/4\Z)^\nu$.
By Remark \ref{sign-mod2}, the reduction modulo~2 of the 
topological signature, i.e., 
$I(X) \to (2\Z)^\nu \to \{\pm1\}^\nu$,
factors as the discriminant $I(X) \to\Pic_G(X)$,
which is onto, followed by the sign map $\Pic_G(X)\to \{\pm1\}^\nu$.
The same is true of the algebraic signature.
Hence we have a commutative diagram:
\begin{equation*}\xymatrix@R=1.5em{
I(V) \ar[d]\ar[r]^-{\onto} & H^1_\et(V,\Z/2)\ar[d]^{\cong}\ar[r]^-{\textrm{sign}} 
                                    & \{\pm1\}^\nu\ar[d]^{=}\\
I(X)\ar[r]^-{\onto} &\H_G^1(X,\Z/2)\ar[r]^-{\textrm{sign}} & \{\pm1\}^\nu.
} \end{equation*}
By inspection, the images of $I(V)$ and $I(X)$ in $\{\pm1\}^\nu$ agree.
\end{proof}

\begin{subremark}
Here are two consequences of Proposition \ref{same-sign}:\\
i) If the signature $WR(X)\to\Gamma$ is onto, so is the signature
$W(V)\!\to\Gamma$. \\
ii) If $4\Z^\nu\subseteq S_a$ then $S_a=S_t$.
\end{subremark}

\goodbreak 
\section{$WR$ is a birational invariant}\label{sec:birational}

Let $V, V'$ be smooth projective surfaces defined over $\R$.
We say that $V$ and $V'$ are {\it (real) birationally equivalent} if
there exists a chain of blow-ups at smooth closed points
(and their inverses, blow-downs) connecting them.
This is equivalent to the assertion that their function
fields are isomorphic (see \cite[II.6.4]{Silhol}).
By a {\it birational invariant} of smooth projective surfaces over $\R$ 
we mean a function that sends real birationally equivalent surfaces 
to isomorphic objects. For example, the number of real components
is a birational invariant.

\begin{theorem}\label{birational}
$WR$ is a birational invariant of smooth projective surfaces over $\R$.
\end{theorem}

\begin{proof}
Since every birational equivalence between smooth projective surfaces
is a composition of blow-ups and
blow-downs at smooth closed points, the result is immediate
from Theorem \ref{blowup} below.
\end{proof}

\begin{subremark}
We recall how to 
topologically construct the blow-up $V'\!\to V$ of a complex algebraic 
surface $V$ at a smooth point $x$ of $V$. Let $D^4$ denote a 
small 4-disk about $x$; the inverse image of $D^4$ in $V'_\top$
is a tubular neighborhood $T$ of the exceptional fiber $\bP^1_\top\cong S^2$,
and the boundary of $T$ is the same as the boundary $S^3$ of $D^4$.
($T$ has Euler class 1.)
Thus $V'_\top$ is the union of $V_\top-D^4$ and $T$ along $S^3$.
\end{subremark}


\begin{theorem}\label{blowup}
Let $V$ be a smooth projective variety 
defined over $\R$, and $V'\to V$ the blowup at a 
smooth closed point $p$. 
Then
$WR(V_\top)\to WR(V'_\top)$ is an isomorphism.
\end{theorem}

\begin{proof}
\addtocounter{equation}{-1}
\begin{subequations}
\renewcommand{\theequation}{\theparentequation.\arabic{equation}}
Set $X=V_\top$, $X'=V'_\top$.
There are two cases; we may blow up either a real point of $V$
(a fixed point $x$ of $X$), or a complex point of $V$
(a conjugate pair of points on $X$).

\smallskip
\paragraph{\it Fixed points.}
Consider first the case where $V'$ is the blow-up of $V$ 
at an $\R$-point. 
That is, $X'$ is obtained by removing an equivariant disk $D^4$ in $X$
about the point and replacing it with 
a tubular neighborhood $T$ of $S^2$. For any equivariant cohomology
theory $H^*$, the excision $H^*(X,D^4)\cong H^*(X',T)$
yields a Mayer-Vietoris sequence. 

Writing $\bP$ for the $G$-space $\C\bP^1$, 
we have a commutative diagram, whose 
rows are Mayer--Vietoris sequences and whose columns are 
the exact Bott sequences (see \cite[Thm.\,6.1]{Schlichting.Fund}
and \cite[(1.5)]{KSW}).
\begin{equation}\label{Bott}
\xymatrix@R=1.5em{
KR_1(\bP,\pt) \ar[r]\ar[d]^{\cong} & KR(X) \ar[r]\ar[d]^{H}
  & KR(X') \ar[r]^-{\mathrm{onto}}\ar[d]^{H} & KR(\bP,\pt)=\Z \ar[d]^{\mathrm{onto}}
\\
GR_1(\bP,\pt) \ar[r]\ar[d]^{0} & GR(X) \ar[r]\ar[d]^{\partial}
  & GR(X') \ar[r]\ar[d]^{\partial} & GR(\bP,\pt)=\Z/2 \ar[d] \\
\uu{0}(\bP,\pt) \ar[r]^{\sigma}\ar[d]^2 & \uu{-1}(X) \ar[r]^u \ar[d]^F
  & \uu{-1}(X') \ar[r]\ar[d]^F & \uu{-1}(\bP,\pt)=0 \ar[d] \\
KR(\bP,\pt) \ar[r]^0\ar@<2ex>[u]^{\tau} & KR_{-1}(X) \ar[r]\ar@<2ex>[u]^{\tau}
  & KR_{-1}(X') \ar[r]\ar@<2ex>[u]^{\tau}  & KR_{-1}(\bP,\pt)=0
}\end{equation}
The Bott sequences in the first and last columns are well known, with 
$KR_1(\bP,\pt)\cong\Z/2$ mapping onto $GR_1(\bP,\pt)\cong\Z/2$,
$KR(\bP,\pt)\cong\Z$ mapping onto $GR_1(\bP,\pt)\cong\Z/2$,
and $\uu{0}(\bP,\pt)=\Z$ mapping to $KR(\bP,\pt)=\Z$ by multiplication by 2.

Since the canonical bundle on $\bP$ lifts to the Real bundle on $X'$
associated to the invertible sheaf $\cO(1)$ on $V'$ \cite[II.7]{Hart}, 
the top right horizontal map is onto, 
and the bottom left horizontal map is zero.
It follows from a diagram chase that the map $WR(X)\to WR(X')$ is onto.
To see that it is an injection, we will show that the map 
$\sigma$ is zero, and hence that the map $u$
is an injection; this suffices, since $WR(X)$ and $WR(X')$ 
are the images of the vertical maps $\partial$ in 
$\uu{-1}(X)$ and $\uu{-1}(X')$.

The map $\tau:KR_n(X)\to\uu{n}(X)$ in \eqref{Bott} is defined as the
composition of the map $KR_n(X)\to {}_{-}GR_n^{[1]}(X)$ in the Bott sequence
with the isomorphism ${}_{-}GR_n^{[1]}\cong\uu{n}$ of the Fundamental
Theorem; see \cite[6.2]{Schlichting.Fund}.  Identifying 
$\uu{0}(\bP,\pt)$ with the relative group associated to
the hyperbolic functor, $\tau$ sends the class of
a Real bundle $E$ to the hyperbolic bundle $E\oplus E^*$, 
equipped with its two Lagrangians $E\oplus0$ and $0\oplus E^*$. 
The composition $F\circ\tau$ sends $[E]$ to $[E]-[E^*]$.
%
Since $[E^*]=-[E]$ in $KR(\bP,\pt)$, 
the composition $F\circ\tau$ on 
the lower left group $KR(\bP,\pt)\cong\Z$ is multiplication by~2;
it follows that $\tau:KR(\bP,\pt)\to \uu{0}(\bP)$ is an isomorphism.
\end{subequations}

\smallskip 
\paragraph{\it Conjugate points}
We now consider the case where $X'$ is the blow-up of $X$ at 
a conjugate pair of points.  In this case, we have a diagram
like \eqref{Bott}, with $(\bP,\pt)$ replaced by $(G\times\bP,G\times\pt)$.
As in \cite[2.4(b) and 2.6]{KSW} (cf.\ Lemma \ref{Arason} above),
$\uu{0}(G\times\bP,G)\cong KO^{-2}(S^2,\pt)$ and
$KR(G\times\bP,G)\cong KU(S^2,\pt)$.
Thus the bottom left and upper right vertical maps are
the familiar maps:
\begin{align*}
\Z=KO^{-2}(S^2,\pt) &\map{2} KU(S^2,\pt)=\Z, \mathrm{~and}\\
\Z=KU(S^2,\pt) &\map{\mathrm{onto}} KO(S^2,\pt)=\Z/2.
\end{align*}
Again, the Real bundle on $X'$ associated to $\cO(1)$ maps to the
canonical complex bundle generating $KU(S^2,\pt)$ so, 
as in \eqref{Bott}, the upper right horizontal map is onto
and the lower left horizontal map is zero.
Again, as in diagram \eqref{Bott}, this implies that 
the map $\sigma$ is zero and hence that $u$ is an injection.
(In fact, $u$ is an isomorphism as $KO^1(S^2,\pt)=0$.)
As before, this suffices to show that $WR(X)\to WR(X')$ 
is an isomorphism.
\end{proof}

\goodbreak
\section{Stiefel--Whitney classes on $KO_G$}\label{sec:SW}


\noindent
If $X$ is a compact $G$-space, let us write $\XG$ for $X\times_GEG$,
so that the (Borel) equivariant cohomology group $H^n_G(X,\Z/2)$
is defined to be $H^n(\XG,\Z/2)$.
Following \cite{AS}, we can identify vector bundles on $\XG$
with equivariant vector bundles on $X\!\times\!EG$;
the pullback of a bundle on $\XG$ is an equivariant vector bundle on
$X\times EG$. As in \cite{AS}, we define $KU(\XG)$ and $KO(\XG)$ to be 
representable $K$-theory:  $KU(\XG)=[\XG,KU]$ and $KO(\XG)=[\XG,KO]$.

The map $X\times EG\to X$ induces natural maps
$KU_G(X)\to KU(\XG)$ and $KO_G(X)\to KO(\XG)$, 
called {\it Atiyah-Segal maps},
since they were first studied by Atiyah and Segal in \cite{AS}.
Atiyah and Segal showed in \cite{AS} that $KO(\XG)$ 
is the completion of $KO_G(X)$ (for finite CW complexes $X$) with
respect to the augmentation ideal of $KO_G(\pt)\cong RO(G)$.

\begin{defn}\label{def:w_n}
The maps $w_n:KO_G(X)\to \H^n_G(X,\Z/2)$, obtained by
composing the Atiyah--Segal map with the usual Stiefel--Whitney classes
on $X_G$, are called the equivariant Stiefel--Whitney classes.
\end{defn}

The equivariant Stiefel--Whitney class $w_1$ agrees on $KO_G(X)$ with
the discriminant $w_1$ defined in Definition \ref{def:w1}.
To see this, note that $\Pic_G(X)\cong\Pic_G(X\!\times\! EG)$, and
elements of $\Pic_G(X\!\times\! EG)$ correspond to
rank~1 $\R$-linear vector bundles on $\XG$, which are classified
by the group $H^1(\XG,\Z/2)$; see \cite[I.4.11]{WK}.

The equivariant Stiefel--Whitney class $w_1$ does not factor
through $WR(X)$, as the example $X=\pt$ shows 
(see Example \ref{trivial action}). 
However, we showed in \cite[5.2]{RGB} that
the restriction of $w_1$ to $I(X)$, the kernel of 
$KO_G(X)\to\H_G^0(X,\Z/2)$, does induce a map $I(X)\to\H_G^1(X,\Z/2)$,
which we shall call $w_1$ by abuse of notation.
The following result was proven in \cite[Thm.\,5.3]{RGB}.

\begin{theorem}\label{w1:factors}
The algebraic discriminant of a smooth variety $V$ factors as
\[
I(V) \to I(V_\top)\ \map{w_1}\ \H^1_G(V_\top,\Z/2)\cong H^1_\et(V,\Z/2).
\]
\end{theorem}

Write $I_2(X)$ for the kernel of $w_1:I(X)\to \H^1_G(X,\Z/2)$.

\begin{corollary}\label{iso-modI2}
$W(V)/I_2(V)\smap{\cong} WR(X)/I_2(X)$, where $X\!=\!V_\top$.
\end{corollary}

\begin{proof}
By Definition \ref{def:w1} and Theorem \ref{w1:factors},
the isomorphism of Lemma \ref{In/In+1} is the composition
\[
I(V)/I_2(V)\!\to\!I(X)/I_2(X)\,\map{\cong}\H^1_G(X,\Z/2).
\]
The extension from $I(V)\cong I(X)$ to $W(V)\cong W(X)$ is routine.
\end{proof}

As noted in Definition \ref{def:w1},
the rank and $w_1$ give a ring homomorphism
$KO_G(X)\to\Z\times \H^1_G(X,\Z/2)$.  It is onto, because $\Pic_G^0(X)$
maps onto $\H^1_G(X,\Z/2)\cong\! H^1(X/G,\Z/2)$ (see \ref{def:PicG}).

\begin{lemma}\label{Z/4} 
Suppose that $G$ acts freely on a connected space $X$.
Then the rank and $w_1$ induce a non-trivial extension
\[ 
0 \to H^1(X/G,\Z/2) \to WR(X)/I_2(X)\ \map{\mathrm{rank}}\ \Z/2\to 0.
\]
This is the extension 
such that $2\cdot1$ is $[-1]$, where '1' is the image of 
the unit of the ring $WR(X)$, and 
$[-1]\in H^1(X/G,\Z/2)\cong\Pic_G^0(X)$ is the class of 
the line bundle $X\times\R(1)$.
\end{lemma}

\begin{proof}
The map $H:KR(X)\to KO_G(X)$ is compatible with the filtration 
$F_*$ associated to the Bredon spectral sequence \eqref{Bredon-ss}, and
we have $KO_G(X)/F_2=\Z\times \H^1_G(X,\Z/2)$. Since 
$H^1_G(X;KR^{-1})=0$ we have $KR(X)/F_2\cong\Z$.
By Remark \ref{H(1)}, $(2,[-1])$ is the image under $w_1\circ H$
of the trivial Real line bundle in $\Z\times \H^1_G(X,\Z/2)$. 
Since $[-1]\ne0$ by Remark \ref{Xcovers}, the result follows.
\end{proof}


The Stiefel--Whitney class $w_2:KO_G(X)\to\H_G^2(X,\Z/2)$
does not factor through $WR(X)$ either, because $w_2$
need not vanish on the image of 
$H:KR(X)\to KO_G(X)$. Consider the composition 
\[ \bar{w}_2:WR(X)\to {_2}\H^3_G(X,\Z(1)). \]
of the map $w_2:KO_G(X)\to \H^2_G(X,\Z/2)$, followed by the Bockstein.
We proved the following result in  \cite[5.6 and 5.8]{RGB}.

\goodbreak
\begin{theorem}\label{thm:Hasse}
If $V$ is a smooth variety, then the composition of
\[
GW(V)\to GR(V_\top)\cong KO_G(V_\top)\to WR(V_\top)
\]
with $\bar{w}_2$ agrees with the algebraic Hasse invariant
\ref{def:disc}, followed by the Bockstein:
\[
GW(V)\ \map{\mathrm{Hasse}}\  H^2_\et(V,\Z/2)/(\Pic V/2) 
\ \map{\beta}\ {_2}\H_G^3(V_\top,\Z(1)).
\]
That is, $\beta\,\mathrm{Hasse}(\theta)=\bar{w}_2(\theta_\top)$
for all $\theta\in GW(V)$.
\end{theorem}

\goodbreak\medskip
We end this section with a discussion of equivariant Chern classes.

\begin{defn}\label{Chern}
In \cite{Kahn}, Bruno Kahn defined equivariant Chern classes 
$c_n:KR(X)\to \H^{2n}_G(X,\Z(n))$ for Real vector bundles, with
the first Chern class $c_1$ inducing an isomorphism between the group of
rank~1 Real vector bundles on $X$ and $\H^2_G(X,\Z(1))$, 
where $\Z(1)$ is the sign representation. 
In particular, $c_1:KR(X)\to \H^2_G(X,\Z(1))$ is a surjection.
\end{defn}

The following calculation is taken from
\cite[Th.\,1]{Kahn}. Recall from Definition \ref{Chern} that the
first Chern class $c_1(E)$ of a Real vector bundle $E$ on $X$
is an element of $H^2_G(X,\Z(1))$. If $E$ has rank $d+1$ 
and $\bP(E)$ is the associated projective bundle, then the fiber of
$\bP(E)$ over a point $x\in X$ is the copy of
$\C\bP^d$ corresponding to $\pi^{-1}(x)$.

\begin{theorem}[Kahn]\label{PBF}
Let $E\smap{\pi} X$ be a rank~2 Real bundle on a $G$-space $X$, 
and let $\bP(E)$ denote the associated projective space. 
If $\xi=c_1(E)\in \H^2_G(X,\Z(1))$ is the first Chern 
class of $E$ then for any coefficient system $A$
\[
\H^n_G(\bP(E),A) = \H^n_G(X,A) \oplus \H^{n-2}_G(X,A(1))\cdot\xi.
\]
In particular, $\H^1_G(\bP(E),\Z/2) = \H^1_G(X,\Z/2)$ and
\[
\H^2_G(\bP(E),\Z(1)) = \H^2_G(X,\Z(1)) \oplus \H^0_G(X,\Z).
\]
\end{theorem}

\begin{subremark}\label{H_G(P2)}
Kahn's formula is more general. In particular, 
if $X$ is a point and $E=\cO_X^2$
so that $\bP(E)=\bP^2$, it yields 
\[
\H^n_G(\bP^2,A) = \H^n_G(\pt,A) \oplus \H^{n-2}_G(\pt,A(1))\cdot\xi
\oplus \H^{n-4}_G(\pt,A)\cdot\xi^2.  \]
\end{subremark}

\goodbreak
\section{Algebraic surfaces with no $\R$-points}\label{sec:C}

Let $V$ be an ireducible variety defined over $\R$.\
Recall \cite[Ex.\,II.3.15]{Hart} that $V$ is said to be
{\it geometrically connected} (or {\it geometrically integral})
if $V\times_\R\C$ is connected, i.e., the function field of $V$
does not contain $\C$.  In this case, $V_\top$ is connected.

If $V$ is an irreducible surface defined over $\R$, and $V_\top$ 
is not connected, then $V$ is also defined over $\C$, i.e., 
$V$ is a complex variety. In this case, 
the $G$-space $V_\top=\Hom_{\R}(\Spec(\C),V)$ is 
$G\times V(\C)$, where $V(\C)=\Hom_{\C}(\Spec(\C),V)$ 
is the usual space of complex points.

We proved the following theorem in \cite[Thm.\,7.4]{RGB};
it is a refinement of a theorem of Zibrowius \cite[5.12]{Zibrowius}.
The invariant $\rho$ is the rank of the cokernel of the 
first Chern class $c_1:NS(V)\to H^2(V(\C),\Z)$, and
$p_g$ is the geometric genus;
the inequality $\rho\ge2\,p_g$ in Theorem \ref{pg=0} 
follows from Hodge theory.

\begin{theorem}\label{pg=0}
Suppose that $V$ is a smooth projective surface over $\C$. 
Then there is a split exact sequence
\[
0\to (\Z/2)^{\rho} \to W(V)\to WR(V_\top) \to 0,
\]
where $\rho\ge2\,p_g$, and $\rho=0$ when $p_g=0$.
Thus $W(V)\to WR(V_\top)$ is an isomorphism if and only if ~$p_g=0$.
\end{theorem}

Surfaces with $p_g=0$ include the projective plane $\bP^2$,
rational surfaces, ruled surfaces, K3 surfaces, and Enriques surfaces
(see \cite{Hart}).
It also includes some surfaces of general type, such as Godeaux surfaces,
Burniat surfaces and Mumford's fake projective plane.

\medskip
When $V$ is geometrically connected, i.e., $V_\top$ is connected, 
the following result was proven in \cite[Thm.\,8.5]{RGB}. 
The invariant $\rho_0$ is the rank of $c_1:\Pic(V)\to \H_G^2(X,\Z(1))$, 
and the relation $\rho_0\ge p_g$ follows from Hodge theory.

\goodbreak
\begin{theorem}\label{WR:pg=0}
Let $V$ be a smooth geometrically connected projective surface over $\R$
with no real points. Then there is a split exact sequence
\[
0\to (\Z/2)^{\rho_0} \to W(V)\to WR(V_\top) \to 0,
\]
where $\rho_0\ge p_g$, and $W(V)\to WR(V_\top)$ is
an isomorphism if and only if $V$ has geometric genus~$p_g=0$.
\end{theorem}

The following calculation was given in \cite[8.1]{RGB}.
As in Lemma \ref{WR(graph)}, $\tH^1(X/G,\Z/2)$ denotes the
quotient of $H^1(X/G,\Z/2)$ by the subgroup generated by 
$[-1]=w_1(X\times\R(1))$.

\begin{theorem}\label{WRX-free}
Let $X$ be a connected 4-dimensional $G$-CW complex.
If $G$ acts freely on $X$, then $WR(X)$ is a $\Z/8$-algebra, and
there is an extension:
\[
0 \to {_2}\H^3_G(X,\Z(1)) \to WR(X) ~\smap{}~ \Z/4\times\tH^1(X/G,\Z/2)\to0;
\]
\end{theorem}

In the rest of this section, we describe $WR(V_\top)$ for several 
surfaces with $p_g=0$ and with no real points.

\medskip
\paragraph{\bf Forms of $\bP^1\times\bP^1$}
A variety $V$ is a {\it form} of $\bP^1\times\bP^1$ if
$V\otimes_\R\C\cong \bP^1\times\bP^1$. In this case, $p_g=0$ and
$V_\top$ is $S^2\times S^2$, so $W(V)\cong WR(V_\top)$
by Theorem \ref{WR:pg=0}.
In this case, the calculation of $WR$ is 
dictated by the structure of $H^2(V_\top,\Z)$ as a $G$-module.

\begin{theorem}\label{formsof P1xP1}
Suppose that $G$ acts freely on $X=S^2\times S^2$.

\noindent
(i) If $G$ acts as $-1$ on $H^2(X,\Z)$ then $WR(X)\cong\Z/4$.

\noindent
(ii) If 
$H^2(X,\Z)\cong\Z[G]$ as a $G$-module then 
$WR(X)\cong\Z/4\oplus\Z/2$.

\noindent
(iii) If $H^2(X,\Z)\cong \Z\oplus\Z(1)$ as a $G$-module then
$WR(X)\cong\Z/4\oplus\Z/2$.
\end{theorem}

\begin{subexamples}
Cases (i)--(iii) are the only possibilities, because the only $G$-module
structures on $\Z^2$ are: $\Z^2$, $\Z(1)^2$, $\Z[G]$ and $\Z\oplus\Z(1)$
(see \cite[2.1]{PW1}), and $G$ cannot act trivially on $H^2(X,\Z)$
(by the Lefschetz Fixed Point Theorem).

Let $X$ be the space $S^2\times S^2$, where $S^2=\C\cup\{\infty\}$.
The involutions sending $(z_1,z_2)$ to $(-\bar{z}_1^{-1},-\bar{z}_2^{-1})$,
$(\bar{z}_2^{-1},-\bar{z}_1^{-1})$ and $(-\bar{z}_1^{-1},z_2)$ 
have no fixed points, and yield the $G$-space structures 
$\Z(1)^2$, $\Z[G]$ and $\Z\oplus\Z(1)$.
These illustrate cases (i), (ii) and (iii) of 
Theorem \ref{formsof P1xP1}.
\end{subexamples}

\begin{proof}[Proof of Theorem \ref{formsof P1xP1}]
Since $\pi_1(X)\!=\!0$, 
we have $\pi_1(X/G)=\Z/2$. Hence $H^1(X/G,\Z/2)$ is $\Z/2$.
It follows from Theorem \ref{WRX-free}  
that $WR(X)$ is an extension of $\Z/4$ by ${_2}\H^3_G(X,\Z(1))$.

To determine ${_2}\H^3_G(X,\Z(1))$, we recall from \ref{Borel} that
\[
\H^*_G(X,\Z)\cong\!H^*(X/G,\Z) \quad\textrm{ and }\quad
\H^*_G(X,\Z[G])\cong H^*(X,\Z).
\]
Hence the long exact cohomology sequence associated to 
the exact sequence $0\to\Z\to\Z[G]\to\Z(1)\to0$
of $G$-modules becomes:
\addtocounter{equation}{-1}
\begin{equation}\label{H4-X/G}
0=H^3(X,\Z) \to \H^3_G(X,\Z(1))\to H^4(X/G,\Z)\to H^4(X,\Z).
\end{equation}
Now consider the spectral sequence
$H^p(G,H^q(X,\Z))\Rightarrow H^{p+q}(X/G,\Z)$
associated to $\XG\to BG$. For convenience, we omit the coefficient
when it is $\Z$.  If $t$ is the nonzero element of 
$E_\infty^{2,0}=E_2^{2,0}=H^2(BG)\cong\Z/2$ then except for 
$p=0$, 
cup product with $t$ is an isomorphism $E_2^{p,q}\cong E_2^{p+2,q}$\!,
commuting with the differentials.

\goodbreak
We now consider the possible $G$-module structures on $H^2(X)\cong\Z^2$.

(i) If $H^2(X)\cong\Z(1)^2$ then $H^4(X)\cong\Z$ as a $G$-module.
For even $p>4$, the sequence
\[ 0\to \Z/2=E_3^{p-6,4} \map{d_3} E_3^{p-3,2}=(\Z/2)^2
\map{d_3} E_3^{p,0}=\Z/2 \to 0
\]
must be exact. By $t$--periodicity, the map 
$d_3^{1,2}:E_3^{1,2}\to E_3^{4,0}$
must be onto. It follows that $H^4(X/G)\to H^4(X)$ is an injection
(with cokernel $\Z/2$). 
The exact sequence \eqref{H4-X/G} yields $\H^3_G(X,\Z(1))=0$.

(ii) If $H^2(X)$ is $\Z[G]$, we have $\Z/2\cong H^4(BG)\smap{\cong}H^4(X/G)$. 
This is because $H^4(X)\cong\Z(1)$ so $H^4(X)^G=0$, and
$H^p(G,H^2(X))=0$ for $p>0$. The result follows from \eqref{H4-X/G}.

(iii) If $H^2(X,\Z)\cong\Z\oplus\Z(1)$ then $H^4(X)\cong\Z(1)$ as a $G$-module.
For even $p\ge6$, the sequences
\begin{align*}
 0=E_3^{p-6,4} \map{d_3} E_3^{p-3,2}=\Z/2 &\map{d_3} E_3^{p,0}=\Z/2 \to 0, \\
0\to \Z/2=E_3^{p-5,4} \map{d_3} E_3^{p-2,2}=\Z/2 &\map{d_3} E_3^{p+1,0}=0
\end{align*}
must be exact.  It follows that $H^4(X/G,\Z)\cong\Z/2$.
\end{proof}
\goodbreak

\medskip
\paragraph{\bf The anisotropic quadric}

Let $Q_2$ be the anisotropic quadric surface $x^2+y^2+z^2+w^3=0$.
It is well known that  the topological $G$-space $X$
underlying $Q_2$ is $S^{3,0}\times S^{3,0}$, where $S^{3,0}$ is the
2-sphere with the antipodal involution. 
(See \cite[VI.5]{Silhol} or \cite[Proof of 2.4]{CTSujatha} for example.)

\begin{corollary}
$W(Q_2)\cong WR(S^{3,0}\times S^{3,0})\cong\Z/4$.
\end{corollary}

\begin{proof}
This follows from Theorem \ref{formsof P1xP1}(i) and Theorem \ref{WR:pg=0}.
\end{proof}

\begin{subremark}\label{rem-correction}
The fact that $W(Q_2)\cong\Z/4$ is due to
Parimala \cite[p.\,92]{Parimala2}.
See Theorem \ref{W=WR=Z/4}  
for a short algebraic proof.
\end{subremark}


\medskip
\paragraph{\bf Rational surfaces}
An  algebraic surface defined over $\R$ is called 
a {\it real rational surface} if 
$V_\C$ is birational to the projective plane over $\C$.
For example, $Q_2$ is a rational surface.

If $V$ is a real rational surface with no real points,
then $p_g=0$, and $W(V)\cong WR(V_\top)$ by Theorem \ref{WR:pg=0}.

\begin{theorem}\label{W=WR=Z/4}
Let $V$ be a real rational surface with no real points.
Then $W(V)~\map{\cong}~ WR(V_\top) \cong \Z/4$.
\end{theorem}

\begin{proof}
We calculate $W(V)$ and $WR(V_\top)$ separately,
for comparison purposes.
We begin with the algebraic calculation of $W(V)$.

Sujatha proved in her thesis that $W(V)\cong\Z/4$.
Unfortunately, the published version  \cite[4.2]{Sujatha}
asserts that $W(V) = \Z/4\oplus\Z/2$, due to a typo in the proof
(on top of p.\,100 in {\it loc.\,cit.}): the incorrect formula 
$\Pic(V_\C)^G\cong\Pic(V)\oplus\Z/2$, should say
that the torsionfree group $\Pic(V_\C)^G$ is a nontrivial extension
of $\Br(\R)\cong\Z/2$ by $\Pic(V)$:
\[ 0 \to \Pic(V) \to \Pic(V_\C)^G\ \map{d_3}\ \Br(\R)\to 0. \]
With this correction, Sujatha's proof of \cite[4.2]{Sujatha}
yields 
\[\dim\!H^2(V,\Z/2)\!=\dim(\Pic V)/2 \quad\mathrm{and}\quad
{_2}\!\Br(V)=0
\]
(so $k=0$ in the notation of \cite{Sujatha}), 
and hence $W(V)\cong\Z/4$.

We turn to the topological calculation, setting $X\!=\!V_\top$. 
In this case, $H^1(X,\Z)\!=\!H^3(X,\Z)\!=\!0$,
and $H^2(X,\Z)$ is given by Lemma \ref{H2-c} below.
%
The spectral sequence 
$E_2^{p,q}=H^p(G,H^q(X,\Z))\Rightarrow\H_G^*(X,\Z)$
agrees with the spectral sequence in the proof of 
Theorem \ref{formsof P1xP1}(i), except for the $E_2^{0,2}$ term.  
As this term played no role in the proof of {\it loc.\,cit.},
the same argument (using \ref{WR:pg=0}) 
applies to yield $WR(X)\cong\Z/4$.
\end{proof}

\begin{lemma}\label{H2-c}
Let $V$ be a real rational surface with no real points.
Then there is an integer $c$ such that
$H^2(V_\top,\Z)\cong \Z(1)^2\oplus \Z[G]^c$.
\end{lemma}

\begin{proof}
It is well known that 
$\Pic(V_\C)\cong\Z^2\oplus \Z[G]^c$ as a $G$-module;
see \cite{ParimalaSujatha} or \cite[p.\,46]{Silhol}. 
Since $H^1_\an(X,\cO)=H^2_\an(X,\cO)=0$, the equivariant 
exponential sequence $0\to\Z(1)\to\cO_\an\to\cO_\an^\times\to0$ 
on $X$ implies that as a $G$-module:
\[
\Pic(V_\C) \cong  H^1_\an(X,\cO_\an^\times)
\cong H^2_\an(X,\Z(1)) \cong H^2(X,\Z)\otimes\Z(1).
\qedhere\]
\end{proof}

\begin{subremark}\label{Comessatti}
Every finitely generated torsionfree $G$-module $A$ is isomorphic
to $\Z^a\oplus\Z(1)^b\oplus \Z[G]^c$ for some integers $a,b,c$.
The integer $c$ is called the {\it Comesssatti characteristic} of $A$,
since Comessatti used it to study abelian varieties over $\R$ in
the 1920's; see \cite{C1,C2}.
The integer $c$ in Lemma \ref{H2-c}
 is called the {\it Comessatti characteristic} of $H^2(X)$.
The Comessatti character of $H^d(V_\top,\Z)$, $d=\dim V$,
has played a major role in computations by Silhol \cite[IV.1]{Silhol},
Parimala--Sujatha \cite{ParimalaSujatha}
and the authors \cite[1.1]{PW}, \cite[4.2]{KSW}. 
\end{subremark}

\medskip
\paragraph{\bf Ruled surfaces}
Recall that a {\it ruled surface} over a curve $C$ is a variety of the form
$V=\bP(E)$, where $E$ is a locally free sheaf of rank~2 over $C$.
For example, if $E$ is free then $V=C\times\bP^1$.
Ruled surfaces have $p_g=0$, so Theorem \ref{WR:pg=0} yields
$W(V)\cong WR(V_\top)$ if $C$ has no $\R$-points, since
in that case $V$ has no $\R$-points either.

\begin{theorem}\label{CxP1-free}
Let $C$ be a smooth, geometrically connected projective curve over $\R$ 
of genus $g$.  If $C$ has no real points and 
$V$ is a ruled surface over $C$, then
\[
W(V) \cong WR(V_\top) \cong WR(C_\top) \cong\Z/4\oplus(\Z/2)^g. 
\]
\end{theorem}

\begin{proof}
It suffices to show that $WR(V_\top)\!\cong WR(C_\top)$, because
$W(C)\cong WR(C_\top)\cong\Z/4\oplus(\Z/2)^g$ by Corollary \ref{WR(curve)}.
Because $\H^3_G(C_\top,\\ \Z(1))=0$, we see from the
Projective Bundle Theorem \ref{PBF} that
$$\H^3_G(V_\top,\Z(1))\cong \H^1_G(C_\top,\Z)\cong H^1(C_\top/G,\Z),$$
which is torsionfree. 
Theorem \ref{WRX-free} yields $WR(V_\top)\cong WR(C_\top)$.
\end{proof}

\begin{proof}[Algebraic proof]
It suffices to show that $W(V)\cong W(C)$.
By the algebraic Projective Bundle Theorem 
(which follows from \ref{PBF} and \cite{Cox}),
$H^1_\et(V,\Z/2)=H^1_\et(C,\Z/2)$
and $H^2_\et(V,\Z/2)=H^2_\et(C,\Z/2)\oplus\Z/2$.
In addition, $\Pic(V)\cong\Pic(C)\oplus\Z$. 
Hence $\Br(V)=\Br(C)=0$, and  $W(V)\cong W(C)$
by Corollary \ref{W/I2}.  
\end{proof}

\begin{subremark}
If $C'$ is the (affine) curve obtained by removing $n\ge1$ 
(complex) points from the projective curve $C$, and
$V'=C'\times\bP^1$, then we still have 
$W(V')\cong WR(V')\cong W(C')$, but now (using Remark \ref{affine curve})
$W(V')\cong W(V)\oplus(\Z/2)^{n-1}$.
This follows from Corollary \ref{WR:pg=0}, since $\Br(V')=\Br(C')=0$
(see \cite[3.6]{PW}) and $H^1_\et(V',\Z/2)\cong H^1_\et(C',\Z/2)$.
\end{subremark}

Let $Q_1$ denote the projective curve $x^2+y^2+z^2=0$, which is
a real form of $\C\bP^1$; its underlying $G$-space $Q_{1,\top}$ is 
$S^{3,0}$ ($S^2$ with the antipodal involution), and 
$Q_{1,\top}/G$ is $\R\bP^2$.
Thus $H^1(Q_1,\Z/2)\cong\Z/2$.

Let $C$ be a smooth (geometrically connected)
projective curve over $\R$, and consider the real form 
$V=C\!\times\!Q_1$ of the ruled surface $C_\C\!\times\!\bP^1_\C$.

\goodbreak
\begin{proposition}\label{CxQ}
When $V=C\times Q_1$, and $C(\R)$ has $\nu$ components, 
\[ W(V)\map{\cong} WR(V_\top) \cong \begin{cases}
\Z/4\oplus(\Z/2)^g, & \nu=0; \\ 
\Z/4\oplus(\Z/2)^g\oplus(\Z/4)^{\nu-1}, & \nu>0.
\end{cases}\]
\end{proposition}

\begin{proof}
Since $p_g=0$, this follows from Theorem \ref{WR:pg=0}.
\end{proof}

\begin{examples}
Both $Q_1\times Q_1$ and the ruled surface $Q_1\times\bP^1$
are surfaces with no real points, and $p_g=0$.
Their underlying $G$-space is $X=S^2\times S^2$.
By Theorem \ref{WR:pg=0}, 
$W(V)\cong WR(V_\top)\cong\Z/4$.
\end{examples}


\medskip\goodbreak
\section{Real points}\label{sec:R-points}

If $X^G\ne\emptyset$ and $\dim(X^G)\le3$ then $WR(X)$
is the sum of $\Z^\nu$ (given by the signature,
as described in Lemma \ref{signature}) and a finite 2--group.
This follows from \cite[7.4]{KSW}, since $KO(X^G)$ has this structure.
In particular, if $V$ is a real algebraic surface 
and $V(\R)$ has $\nu>0$ components, then both 
$W(V)$ and $WR(V_\top)$ are the sum of $\Z^\nu$ and a finite 2--group.

To describe the torsion, we use the cellular filtrations on 
$KR$ and $KO_G$ from the Bredon spectral sequence \eqref{Bredon-ss},
defining $F_p WR(X)$ as the image of $F_pKO_G(X)\to WR(X)$.
Thus $F_1WR(X)$ is the torsion subgroup and $WR(X)/F_1\cong\Z^\nu$.

\begin{lemma}
Suppose that $\dim(X)\!<\!8$ and $\dim(X^G)\le3$. Then $F_3WR(X)=0$.
\end{lemma}

\begin{proof}
The groups $E_2^{p,-p}(KO_G)$ are zero if $p=3,5,6,7$. 
The map $E_2^{4,-4}(KR)\to\!E_2^{4,-4}(KO_G)$ is onto by Lemma \ref{H4-iso}.
Since $\dim X\!<\!8$, $E_\infty^{4,-4}$ is a quotient of $E_2^{4,-4}$
for both the $KR$ and $KO_G$ cases, $E_\infty^{4,-4}(KR)$ maps onto
$E_\infty^{4,-4}(KO_G)$, and hence  $F_3WR(X)=0$.
\end{proof}

Since $F_1/F_2\cong E_\infty^{1,-1}$ for both $KO_G$ and $KR$,
we have an exact sequence
\[
E_\infty^{1,-1}(KR)\ \map{H_\infty}\ E_\infty^{1,-1}(KO_G) \to F_1WR(X)/F_2WR(X) \to0.
\]
When $\dim(X)\le4$ and $\dim(X^G)\le2$,
this fits into a diagram with exact rows, in which
$d_2$ and $\bar{d}_2$ are the $E_2$ differentials: 
\begin{equation}\label{eq:Fil12}
\xymatrix@R=1.5em{
0 \to E_\infty^{1,-1}(KR) \ar[d]^{H_\infty}\ar[r] & H^1(X^G,\Z/2)
 \ar[d]^{\textrm{into}}\ar[r]^{d_2} & {_2}H^3_G(X,\Z(1)) \ar[d]^{\mathrm{mod}~2} \\
0 \to E_\infty^{1,-1}(KO_G) \ar[r]\ar[d]^{\textrm{onto}}   & 
   H^1_G(X;KO_G^{-1}) \ar[r]^-{\overline{d}_2}\ar[d]^{\textrm{onto}} &H^3(X/G,\Z/2)\\
F_1WR(X)/F_2 \ar[r]^-{\textrm{onto}} &\quad \Pic^0_G(X). &
}\end{equation}
The bottom horizontal arrow in \eqref{eq:Fil12} is onto because, 
as pointed out in \ref{def:PicG}, $F_1WR(X)$ maps onto $\Pic_G^0(X)$.
The top middle vertical is an injection with cokernel
$H^1(X/G,\Z/2)\cong\Pic_G^0(X)$ by Lemma \ref{H^q_G} and
Definition \ref{def:PicG}.  It follows that $H_\infty$ is also an injection.

By Example \ref{local-coeffs}, $H_G^3(X,\Z/2)\cong H^3(X/G,\Z/2)$;
the right vertical map is reduction mod~2.

\begin{definition}\label{snake}
Let $\Delta$ denote the kernel of the induced surjection
$F_1WR(X)/F_2\to \Pic_G^0(X)$ in the bottom row of \eqref{eq:Fil12}.

By the snake lemma applied to the columns of \eqref{eq:Fil12},
$\Delta$ is the image in $H^3(X,\Z(1))$ of the
group of all $a$ in $H_G^1(X^G\!,\Z/2)$ which
vanish in $H^3(X/G,\Z/2)$ under the map $\overline{d}_2$ of
\eqref{eq:Fil12}, 
i.e., $\Delta$ is
\[
\{ a\in \mathrm{image}\bigl[H^1(X^G\!,\Z/2)\smap{d_2} H_G^3(X,\Z(1))\bigr] :
\bar{a}\!=0~\text{in}~H^3(X/G,\Z/2) \}.
\]
\end{definition}

\begin{subremark}\label{snake+}
Since the image of $d_2$ has exponent~2, $\Delta$ is also the intersection
of $\mathrm{image}(d_2)$ and $2\cdot H_G^3(X,\Z(1))$.

Clearly, $\Delta=0$ if $H_G^3(X,\Z(1))$ has no 2-torsion.
We also have $\Delta=0$ if the even Tate cohomology of $\Z/2$
acting on $H_G^3(X,\Z(1))$ is trivial.
\end{subremark}

\begin{lemma}\label{F2WR}
When $\dim(X)\le4$ and $\dim(X^G)\le2$,
\[ F_2WR(X)\! \cong F_2KO_G(X)/H(F_2KR(X))\!
\cong  E_\infty^{2,-2}(KR)/H(E_\infty^{2,-2}(KR)).
\]
\end{lemma}

\begin{proof}
Because $H_\infty$ is an injection in diagram \eqref{eq:Fil12}, 
the intersection
$H(F_1KR(X))\cap F_2KO_G(X)$ is the image of $F_2KR(X)$.
\end{proof}

\begin{lemma}\label{Fil12WR}
If $\dim(X)\le4$ and $X^G$ has $\nu>0$ components,
of dimension $\le\!2$, then there is an isomorphism
${_2}H^3_G(X;\Z(1))\smap{\cong}\! F_2 WR(X)$.
\end{lemma}

\begin{proof}
By Lemma \ref{H^q_G} and Corollary \ref{KR(2)},
and Lemma \ref{F2WR},
we also have a commutative diagram with exact columns:
\begin{equation*}\xymatrix@R=1.5em{
(\Z/2)^\nu \ar[d] \ar[r]^-{d_2} & H^2_G(X;\Z(1))\oplus H^2(X^G,\Z/2)
          \ar[d] \ar[r] &E_\infty^{2,-2}(KR) \ar[d]^{H_\infty^{2,-2}} \to0 \\ 
(\Z/2)^{1+\nu}\ar[d]^{\textrm{onto}}\ar[r]^-{d_2} & H^2(X/G,\Z/2)\oplus H^2(X^G,\Z/2)
   \ar[d]^{\textrm{onto}} \ar[r] &  E_\infty^{2,-2}(KO_G)\ar[d]^{\textrm{onto}} \to\!0 \\
(\Z/2) \ar[r] & {_2}H^3_G(X;\Z(1))
\ar[r] &\quad  F_2 WR(X)\! \to\!0. 
}\end{equation*}
A comparison with the corresponding diagram when $X$ is a point
shows that the bottom left horizontal map is zero and hence that we have
${_2}H^3_G(X;\Z(1)) \cong F_2 WR(X)$.
\end{proof}

Recall that the group $\Delta$ is defined in \ref{snake}.

\begin{theorem}\label{WR-X/G}
If $X$ is 4-dimensional and $X^G$ has $\nu>0$ components, 
of dimension $\le2$, the image $S_t$ of the signature $WR(X)\to\Z^\nu$
is a subgroup of 2-primary index, and there is an exact sequence
\[
0\to {_2}H^3_G(X;\Z(1)) \to WR(X)
\to S_t\oplus H^1(X/G,\Z/2)\oplus \Delta \smap{} 0.
\]
In particular, if $H^3_G(X;\Z(1))$ has no 2-torsion, then
\[  WR(X) \cong S_t \oplus H^1(X/G,\Z/2). \]
\end{theorem}


\begin{proof}
Given Lemma \ref{Fil12WR}, we need to determine the kernel
of the surjection $F_1/F_2WR(X)\to\Pic_G^0(X)$ in \eqref{eq:Fil12}.
Since $H_G^1(X;KO_G^{-1})$ has exponent~2, the 
subgroup $E_\infty^{1,-1}(KO_G)$ and its quotient 
$F_1/F_2WR(X)$ also have exponent~2.  
This implies that the extension has a (non-canonical) splitting: 
$F_1/F_2 WR(X)\cong H^1(X/G,\Z/2)\oplus \Delta$.

Finally, if $H^3_G(X;\Z(1))$ has no 2-torsion, then $\Delta = 0$.
\end{proof}

When $V$ is a smooth projective variety over $\R$ and $X=V_\top$, 
we can use the isomorphisms
$H_\et^*(V,\Z/2)\cong \H_G^*(X,\Z/2)$ to compare $W(V)$ and $WR(X)$.
Recall from Section \ref{sec:Pic_G} that the image $S_a$ of
the algebraic signature $W(V)\to\Z^\nu$ is a subgroup of the image
$S_t$ of the topological signature $WR(X)\to\Z^\nu$.

Recall that $\rho_0$ denotes the rank of 
$c_1:\Pic(V)\otimes\Q\to \H_G^2(X,\Q(1))$; we have
$\rho_0\ge p_g$ by Hodge theory, and if $p_g(V)=0$ then $\rho_0=0$.
The group $\Delta$ is defined in \ref{snake}. 
(See also Remark \ref{snake+}.)

\goodbreak
\begin{theorem}\label{thm:W-WR}
Let $V$ be a smooth geometrically connected  projective surface 
defined over $\R$, with $V(\R)\ne\emptyset$.
Then there is an exact sequence
\[
0 \to (\Z/2)^{\rho_0} \to W(V) \to WR(V_\top) \to \Delta\oplus(S_t/S_a) \to 0.
\]
\end{theorem}
\goodbreak

\begin{proof}
As the signature on $W(V)$ factors through the signature on $WR(V_\top)$,
we are reduced to a comparison of their torsion subgroups
$I(V)_\tors$ and $F_1WR(V_\top)$.
By Proposition \ref{Itorsion}, which is due to Sujatha, 
\[
I(V)_\tors/I_2(V)_\tors\cong H^0_\tors(V,\cH^1)\ \mathrm{ and }\
I_2(V)_\tors\cong H^0_\tors(V,\cH^2).
\]
By definition, $H^0_\tors(V,\cH^1)$ and $H^0_\tors(V,\cH^2)$ are the 
kernels of the stabilization maps from
$H^0(V,\cH^1)\cong H_\et^1(V,\Z/2)$ and $H^0(V,\cH^2)\cong {_2}\Br(V)$
to $H_\et^3(V,\Z/2)\cong(\Z/2)^\nu$. 


Recall from \cite[6.3]{RGB} that when $V$ is a smooth 
geometrically connected projective variety over $\R$ then 
there is a split exact sequence
\begin{equation}\label{Krasnov}
0\to (\Q/\Z)^{\rho_0} \to \Br(V) \map{\beta} {}_\tors\H_G^3(X,\Z(1)) \to 0,
\end{equation}
where $\beta$ is induced from the Bockstein on $\H_G^2(X,\Z/2)$,
and $X=V_\top$. 
Hence $(\Z/2)^{\rho_0}$ is the kernel of 
$\beta: {_2}\Br(V)\to\ {_2}\H_G^3(X,\Z(1))$; 
by Theorem \ref{thm:Hasse}, $(\Z/2)^{\rho_0}$ is the kernel of
$I_2(V)_\tors \to  F_2WR(X)$.

It remains to show that the induced map from
$I(V)_\tors/I_2(V)_\tors \cong\! H^0_\tors(V,\cH^1)$ to $F_1WR(X)/F_2WR(X)$ 
is an injection with cokernel $\Delta$.
As we saw in Section \ref{sec:Pic_G},
it follows from 
Definition \ref{def:PicG} that 
\[ H^0_\tors(V,\cH^1) \cong \Pic_G^0(X) \map{\cong} H^1(X/G,\Z/2), \]
where the second map is $w_1$.
By Theorem \ref{w1:factors} and Corollary \ref{iso-modI2},
this is the same as the discriminant on $I(V)$.
It follows from Theorem \ref{WR-X/G} that the cokernel of
$H^0_\tors(V,\cH^1) \to WR(X)/F_2WR(X)$ is $\Delta$.
\end{proof}

\begin{example}\label{ExE}
Consider the real surface $V=E\times E$, where $E$ is 
an elliptic curve with 2 real connected components. 
We saw in Example \ref{W(ExE)} that 
$W(V)\cong\Z^4\oplus(\Z/4)^3\oplus\Z/2$.
We claim that $WR(V_\top)\cong\Z^4\oplus(\Z/2)^3$ so that
$W(V)\!\to\!WR(V_\top)$ is a surjection with kernel $(\Z/2)^4$.

To compute $WR(V_\top)$ we choose a basepoint fixed by the involution,
set $Y=E_\top$ and observe that since $V_\top=Y\times Y$ we have
a split exact sequence
\[
0 \to KR(Y\!\wedge Y)\to KR(Y\!\times Y) \to KR(Y,\pt)\oplus KR(Y,\pt),
\]
and similarly for $KO_G(Y\!\times\!Y)$. Thus we have
\[
WR(V_\top)\cong WR(Y\wedge Y)\oplus WR(Y,\pt)\oplus WR(Y,\pt)
\]
Finally, we use the identity $WR(Y,\pt)\cong\Z\oplus(\Z/2)$ 
of Corollary \ref{WR(curve)}.

In the same way, $Y=S^1\times S^{1,1}$, where $S^{1,1}$ is the unit circle
in $\C$. Therefore $WR(Y\wedge Y)$ decomposes as the sum of
$WR(S^2)=\Z\oplus(\Z/2)$ and the groups $WR(S,\pt)$ for $S$ on the 
ordered list
\[ 
S^{2,3}, S^{2,2}, S^{2,2}, S^{1,3}, S^{1,3}, S^{1,2}, S^{1,2}, S^{2,1}.
\]
Since $S^{1,2}\cong\C\bP^1$, we see from Corollary \ref{WR(curve)} 
that $WR(S^{1,2})=\Z$. We also have $WR(S^{2,1})=\Z^2$ by
Example \ref{WR(2-sphere)}.  
A similar detailed computation, mentioned in Example \ref{WR(2-sphere)},
shows that
\[
WR(S^{2,3})= WR(S^{2,2}) = WR(S^{1,3}) = \Z.
\]
It follows that $WR(Y\times Y) \cong \Z^4\oplus(\Z/2)^3$.
Thus $W(V)\to WR(V_\top)$ is a surjection with kernel $(\Z/2)^4$.
\end{example}

\goodbreak
\medskip{\it Ruled Surfaces}\smallskip

We return to {\it ruled surfaces} over a curve $C$, i.e., varieties
of the form $\bP=\bP(E)$, where $E$ is a rank~2 algebraic line bundle over $C$.
Theorem \ref{CxP1-free} describes the case when $C(\R)=\emptyset$.

\goodbreak
\begin{theorem}\label{ruled}
Let $C$ be a curve defined over $\R$ 
with $\nu>0$ real components, and genus $g$.
If $\bP=\bP(E)$ is a ruled surface over $C$, then 
\[
WR(\bP_\top) \lmap{\cong} WR(C_\top) \cong \Z^\nu\oplus H^1(C_\top/G,\Z/2).
\]
If $C$ is smooth projective, then we also have
$W(\bP) \smap{\cong} WR(\bP_\top)$.
\end{theorem}
\goodbreak

\begin{proof}
Let $X$ and $Y$ denote the $G$-spaces underlying $\bP=\bP(E)$
and $C$, respectively. 
Theorem \ref{WR-X/G} yields $WR(Y)\cong \Z^\nu\oplus H^1(Y/G,\Z/2)$.
By the Projective Bundle Theorem \ref{PBF}, 
$\H_G^3(X,\Z(1))\cong \H_G^1(Y,\Z)$. This group is defined to be
$H^1(Y_G,\Z)$, which is torsionfree, so by
Theorem \ref{WR-X/G} 
we have $WR(X)\cong \Z^\nu\oplus H^1(X/G,\Z/2)$. 
Since $H^1_\et(\bP,\Z/2)\cong H^1_\et(C,\Z/2)$, 
and $H^0(X^G)\cong H^0(Y^G)$, the exact sequence in
Example \ref{Borel} implies that 
$H^1(X/G,\Z/2)\cong H^1(Y/G,\Z/2)$. 
Hence $WR(X)\cong WR(Y)$.

When $C$ is smooth projective, $W(C)\cong WR(C_\top)$ by
Corollary \ref{WR(curve)} and $W(\bP)\cong W(C)$ by Lemma \ref{ruled-W}.
This gives the second assertion.
\end{proof}

\begin{lemma}\label{ruled-W}
Let $C$ be a smooth 
projective curve over $\R$,
with $\nu>0$ real components, and genus $g$.
If $\bP=\bP(E)$ is a ruled surface over $C$, then
$W(\bP) \cong W(C) \cong \Z^\nu\oplus(\Z/2)^g$.
\end{lemma}

\begin{proof}
Now suppose that $C$ is smooth and projective;
we need to show that $W(C)\to W(\bP)$ is an isomorphism.
We have $H^2_\et(\bP,\Z/2)\cong H^2_\et(C,\Z/2)\oplus\Z/2$
by the Projective Bundle Theorem \cite[VI.10.1]{Milne}.
Since $\Pic(\bP)\cong\Pic(C)\oplus\Z$, we have 
${_2}\Br(\bP)= {_2}\!\Br(C)=\Z/2^\nu$. As in Corollary \ref{WR(curve)},
$I_2(C)=\Z/2^\nu$; since $I_2(\bP)\subseteq\, {_2}\Br(\bP)$, 
we see that $I_2(\bP)/I_2(C)$ is 2-torsionfree.

By Sujatha's formulas listed in Remark \ref{def:j,k}, 
$k=\dim H^0_\tors(-,\cH^2)$ is the same for $C$ and $\bP$.
Similarly, $H^1_\et(\bP,\Z/2)\cong H^1_\et(C,\Z/2)$, and
$j=\dim H^0_\tors(-,\cH^1)$ is the same for $C$ and $\bP$.
Thus $W(C) = W(\bP)$, since both groups have order $2^{j+k}$
by  Remark \ref{def:j,k}. 
\end{proof}

\begin{subremark}
Charles Walter (unpublished \cite{Walter}) proved that for $Y$ smooth of
dimension $\le2$, and any projective bundle $\bP$ over $Y$,
we have 
$W(Y)\cong W(\bP)$.
\end{subremark}

\medskip\goodbreak
\section{3-folds}\label{sec:3-folds}

In this section, we collect information about smooth 3-folds
that follows from the techniques in this paper. We begin by
giving a short proof of a result of Parimala \cite{Parimala}.
Recall from \eqref{BO-seq} that $\cH^n$ is the Zariski sheaf 
associated to the presheaf $U\!\mapsto\! H^n_\et(U,\Z/2)$.

\begin{theorem}[Parimala]\label{3folds-fg}
Let $V$ be a smooth 3-dimensional algebraic variety over $\R$ or $\C$.
Then the group $W(V)$ is finitely generated if and only if
$CH^2(V)/2$ is finitely generated.
\end{theorem}

\begin{proof}
By Cox' theorem, both $H^1_\et(V,\Z/2)$ and $H^2_\et(V,\Z/2)$ are finite.
Hence Lemma \ref{In/In+1} and Remark \ref{I_2/I_3} 
show that $W(V)/I_3(V)$ is finite.
In addition, $W(V)/\mathrm{torsion}$ is $\Z^\nu$, where $V(\R)$ has $\nu$
connected components, and $I_4(V)$ is torsionfree (by \ref{exp.2^d}).
Hence $W(V)$ is a finitely generated group if and only if
$I_3(V)/I_4(V)\cong H^0(V,\cH^3)$ is finite. 
Since $H^2(V,\cH^2)\cong CH^2(V)/2$, and $H^p(V,\cH^q)=0$
for $p>q$ \cite{BlochOgus}, the coniveau
spectral sequence $H^p(V,\cH^q)\Rightarrow H^{p+q}_\et(V,\Z/2)$
yields the exact sequence
\[
H^3_\et(V,\Z/2)\to H^0(V,\cH^3)\ \map{d_2}\ CH^2(V)/2 \to H^4_\et(V,\Z/2).
\]
As the outer terms are finite, it follows that $H^0(V,\cH^3)$ is finitely
generated if and only if $CH^2(V)/2$ is finite.
\end{proof}

\begin{subremark}
If $V$ is a 3-fold defined over $\C$, the Witt group $W(V)$ is an
algebra over $W(\C)=\Z/2$, and no group extension issues arise.
\end{subremark}

For $WR(V_\top)$ of 3-folds, we need a new ingredient:
the cohomology operation $\tilde\beta Sq^2$
from $H^2(X,\Z/2)$ to $H^5(X,\Z)$,
where $\tilde\beta$ is the integral Bockstein.
By \cite[5.9]{RGB}, the kernel of $\tilde\beta Sq^2$ is
the set of all Stiefel--Whitney classes $u=w_2(E)\in H^2(Y,\Z/2)$.

Since $H^5(\C\bP^\infty\!,\Z)=0$ and $K(\Z,2)=\C\bP^\infty$,
$\tilde\beta Sq^2$ vanishes on the mod~2 reduction of $H^2(X,\Z)$.
Hence the following definition makes sense.

\begin{definition}\label{barH2}
We define $\bar{H}^2(Y)$ to be the kernel of the cohomology operation
${_2}H^3(Y,\Z)\to H^5(Y,\Z)$ induced by $\tilde\beta Sq^2$.
\end{definition}

The following result and its proof is a modification of 
the argument we gave in \cite[7.1]{RGB} for 4-dimensional spaces.

\begin{theorem}\label{C-3fold}
If $Y$ is a connected 6-dimensional CW complex, 
\[
WR(G\times Y) \cong \Z/2\times H^1(Y,\Z/2)\times \bar{H}^2(Y).
\]
\end{theorem}

\begin{proof}
Recall that $WR(G\times Y)$ is the cokernel of $KU(Y)\to KO(Y)$.
We compare the filtrations on $KU(Y)$ and $KO(Y)$ associated to
their Atiyah-Hirzebruch spectral sequences.
Since the quotient $KU(Y)/F_2KU(Y) = \Z$ injects into 
\[ KO(Y)/F_2KO(Y)=\Z\times H^1(Y,\Z/2), \]
we see that the cokernel of $F_2KU(Y)\to F_2KO(Y)$ is the kernel
of the surjection $WR(G\times Y)\to\Z/2\times H^1(Y,\Z/2)$.
Since $E_2^{p,-p}(KO)=0$ for $p=3$ and $p>4$, and the map
\[
E_\infty^{4,-4}(KU)\cong H^4(Y,\Z) \to H^4(Y,\Z)\cong E_\infty^{4,-4}(KO)
\]
is onto (because $KU^{-4}\to KO^{-4}$ is onto), 
we are left to consider $F_2/F_3=E_\infty^{2,-2}$ for $KU(Y)$ and $KO(Y)$.

The differential $d_3:E_2^{2,-2}(KU)\to E_2^{5,-4}(KU)$ is
a cohomology operation $H^2(Y,\Z)\to H^5(Y,\Z)$. 
As noted above, it must be zero,
so $E_\infty^{2,-2}(KU)\cong H^2(Y,\Z)$.
Similarly, the differential $d_3:H^2(Y,\Z/2)\to H^5(Y,\Z)$
is the cohomology operation $\tilde\beta Sq^2$, so 
$E_\infty^{2,-2}(KO)$ is the subgroup of Stiefel--Whitney classes
$w_2(E)$ in $H^2(Y,\Z/2)$, and the cokernel of 
$E_\infty^{2,-2}(KU)\to E_\infty^{2,-2}(KO_G)$ 
is the group $\bar{H}^2(Y)$ of Definition \ref{barH2}.
\end{proof}

\begin{corollary}
Let $V$ be a smooth 3-fold over $\C$, and set $Y=V(\C)$.
If $H^5(Y,\Z)\!$ is a torsionfree group, then
\[ WR(V_\top) \cong \Z/2\times H^1(Y,\Z/2)\times {_2}H^3(Y,\Z). \]
\end{corollary}

\begin{proof}
The image of $\tilde\beta Sq^2:H^2(Y,\Z/2)\to H^5(Y,\Z)$ has exponent~2.
It must be zero when $H^5(Y,\Z)$ is torsionfree.
Now apply \ref{C-3fold} to $G\times Y$.
\end{proof}

\begin{example}\label{ex:Totaro}
In particular, if $V$ is an abelian 3-fold over $\C$, then 
$H^1(V(\C),\Z/2)\cong(\Z/2)^6$ and $WR(V_\top)\cong(\Z/2)^7$\!.
This contrasts with Totaro's result \cite{Totaro} 
that $W(V)$ can be infinite for very general $V$.
\end{example}

We now consider a connnected 6-dimensional complex over $\R$
on which $G$ acts freely.
We need the following modification of Definition \ref{barH2}.
We will see in the proof of Theorem \ref{R-3fold} below that
the composition of 
$\H_G^2(X,\Z(1))\!\cong\! H^2(X/G,\Z(1))\to H^2(X/G,\Z/2)$ 
with $\tilde\beta Sq^2$ is zero.

\begin{definition}\label{barH2'}
When $G$ acts freely on $X$ and $Y\!=\!X/G$, 
we define $\bar{H}_G^2(X)$ to be the kernel of the 
cohomology operation ${_2}H^3(Y,\Z(1))\!\to H^5(Y,\Z)$ 
induced by $\tilde\beta Sq^2$.
\end{definition}

Here is the analogue of Theorem \ref{WR:pg=0} for 3-folds over $\R$.
It is a modification of the argument we gave in \cite[7.1]{RGB} 
for 4-dimensional spaces; see Theorem 8.1 of \cite{RGB}.
As in Lemma \ref{WR(graph)}, $\tH^1(X/G,\Z/2)$ denotes the
quotient of $H^1(X/G,\Z/2)$ by the subgroup generated by 
$[-1]=w_1(X\times\R(1))$.

\begin{theorem}\label{R-3fold}
Let $X$ be a connected 6-dimensional $G$-CW complex.
If $G$ acts freely on $X$, then $WR(X)$ is an extension:
\[
0 \to \bar{H}_G^2(X) \to WR(X) \to \Z/4\times\tH^1_G(X/G,\Z/2) \to0,
\]
%
In particular, $WR(X)$ is a $\Z/8$-algebra.
\end{theorem}

\begin{proof}
Since $X$ is a free $G$-space, $KO_G(X)=KO(X/G)$.
To determine $KR(X)$, it is convenient to write $Y$ for $X/G$.
As in the proof of Theorem \ref{C-3fold}, we work with the
filtration coming from the Bredon spectral sequence.
By Lemma \ref{H^q_G}, $E_2^{1,-1}(KR)\cong H^1(X^G,\Z/2)=0$.
Thus $KR(X)/F_2KR(X)=\Z$ injects into 
\[ KO_G(X)/F_2KO_G \cong KO(Y)/F_2KO(Y)=\Z\times H^1(Y,\Z/2), \]
and the kernel of the surjection 
$WR(X)\to \Z/4\times \tH^1(Y,\Z/2)$ is the
cokernel of $F_2KR(X)\to F_2KO_G(X)\cong F_2KO(Y)$.

Since $E_2^{p,-p}(KO)=0$ for $p>4$, $F_5KO(Y)=0$ and
$F_4KO(Y)\cong E_\infty^{4,-4}$.
For both $KR$ and $KO_G$, $E_3^{3,-3}=0$ and
the group $E_\infty^{4,-4}$ 
is a quotient of $H^4(Y,\Z)$. In addition, the map 
\[
E_2^{4,-4}(KR)\cong H^4(Y,\Z) \to H^4(Y,\Z)\cong E_2^{4,-4}(KO_G)
\]
is an isomorphism (because $KU^{-4}\to KO^{-4}$ is). Thus
$F_3KR(X)$ maps onto $F_3KO(Y)$, i.e., $F_3WR(X)=0$.
We are left to consider $F_2/F_3=E_\infty^{2,-2}$ 
for $KR(X)$ and $KO(Y)$.  We have a commutative diagram
\begin{equation}\label{E(2,-2)}
\xymatrix@R=1.5em{
0 \to E_\infty^{2,-2}(KR)\ar[d]\ar[r]&
H^2(Y,\Z(1))\ar[d]\ar[r]^{d_3}& H^5(Y,\Z)\ar[d]^{\cong} \\
0 \to E_\infty^{2,-2}(KO_G)\ar[r]& H^2(Y,\Z/2)\ar[r]^{d_3}& H^5(Y,\Z).}
\end{equation}
%
As in the proof of Theorem \ref{C-3fold}, the differential
$d_3:H^2(Y,\Z/2)\to H^5(Y,\Z)$
is the cohomology operation $\tilde\beta Sq^2(u)$, so its kernel
$E_\infty^{2,-2}(KO)$ is the subgroup 
$\{u\in H^2(Y,\Z/2)\ \vert\ \tilde\beta Sq^2(u)=0\}$.

We claim that the differential
$d_3:H^2(Y,\Z(1))\to H^5(Y,\Z)$ is zero in the top row in \eqref{E(2,-2)},
so that $E_\infty^{2,-2}(KR)\cong E_2^{2,-2}(KR)= H^2(Y,\Z(1))$.
This will imply that the cokernel of 
\[ E_\infty^{2,-2}KR(X)\to E_\infty^{2,-2}KO_G(X) \]
is the subgroup $\bar{H}_G^2(Y)$ of Definition \ref{barH2'},
finishing the proof.

The Real Chern class $c_1$ of Definition \ref{Chern} is the
composition of the map 
$F_2KR(X)\to E_\infty^{2,-2}\subseteq H_G^2(X,\Z(1))\cong H^2(Y,\Z(1))$ 
with the isomorphism $H_G^2(X,\Z(1))\cong\H_G^2(X,\Z(1))$
of \ref{ex:coeff}(c). Since $c_1$ is onto, the differential
$d_3:H^2(Y,\Z(1))\to H^5(Y,\Z)$ is zero, as claimed,
\end{proof}

\begin{subremark}
The fact that $WR(X)$ is a $\Z/8$--algebra can also be deduced from 
the considerations of Section 7 in \cite{KSW}, which are valid 
for any CW-complex with a free $G$-action, the torsion being a 
specific function of the dimension. 
\end{subremark}

Now suppose that $V$ is a 3-fold over $\R$ with no real points.

\begin{corollary}\label{WR-3fold}
If $V$\! is a geometrically connected 3-fold over $\R$, 
with no $\R$-points, then
$WR(V_\top)$ is an extension of $\Z/4\times\tH^1_\et(V,\Z/2)$ by
the subquotient $\bar{H}_G^2$ of $H^2_\et(V,\Z/2)$ defined in
\ref{barH2'}.

The kernel and cokernel of $W(V)\to WR(V_\top)$ have exponent~2, and 
are the same as the kernel and cokernel of $I_2(V)\to\bar{H}_G^2$.
\end{corollary}

\begin{proof}
Setting $X=V_\top$, the first assertion is Theorem \ref{R-3fold}. 
Combining this with Corollary \ref{iso-modI2},
we get a commutative diagram with exact rows:
\begin{equation*}\xymatrix@R=1.5em{
0 \to I_2(V)\quad\ar[d]\ar[r]& W(V)\ar[d]\ar[r] 
                  & \Z/4\times \tH_\et^1(V,\Z/2)\ar[d]^{\cong} \to 0 \\
\kern-19pt 0 \to \bar{H}_G^2 \ar[r] & WR(X)\ar[r]
                  & \Z/4\times \tH_G^1(X,\Z/2) \to 0.}
\end{equation*}
\noindent The result follows from the snake lemma.
\end{proof}

We now turn our attention to the anisotropic quadric 3-fold $Q_3$
defined by $\sum_{i=0}^4 x_i^2=0$.
For notational simplicity, let $X$ denote the $G$-manifold $Q_{3,\top}$.
Colliot-Th\'el\`ene and Sujatha \cite[3.4]{CTSujatha}
proved that $W(Q_3)\!\cong\Z/8$, with $I_2(Q_3)\cong\Br(Q_3)\cong\Z/2$; 
we will show that $WR(X)\cong\Z/8$ as well.

\goodbreak
\begin{lemma}\label{Br(Q3)}
$\Z/2\cong\Br(Q_3) \to {_2}\H_G^3(X,\Z(1))$ is an isomorphism.
\end{lemma}

\begin{proof}
Colliot-Th\'el\`ene and Sujatha observed in \cite[p.\,9]{CTSujatha}
that $\Pic(Q_3)=\Z$ and $\Br(Q_3)\cong\Br(\R)\cong\Z/2$,
and observed that $G$ acts trivially on $\Pic(Q_3)$.
Set $m=2^k$.
From Kummer Theory, we deduce that $H_\et^p(Q_3,\mu_m)$ is
$\Z/2$ for $p=0,1$, and $\Z/m\oplus\Z/2$ for $p=2$.

By Cox'\ theorem, we can identify $H^p_\et(Q_3,\mu_m)$
with $\H_G^p(X,\Z(1)/m)$. Moreover, $\H_G^0(X,\Z(1))=0$,
and the groups $\H_G^p(X,\Z(1))$ are finitely generated.
Using the exact sequences 
\[
0\to \H_G^p(X,\Z(1))/m \to H^p_\et(Q_3,\mu_m) \to {_m}\H_G^{p+1}(X,\Z(1))\to 0,
\]
we see that $\H_G^1(X,\Z(1))=\Z/2$ and $H_G^2(X,\Z(1))=\Z$. From
\[
0 \to \H_G^2(X,\Z(1))/2 \to \H_G^2(X,\Z/2) 
\map{\partial} {_2}\H_G^3(X,\Z(1))\to 0
\]
we see that 
${_2}\H_G^3(X,\Z(1))=\Z/2$.
Since $\partial$ factors through the surjection
$H^2_\et(Q_3,\Z/2)\to {_2}\Br(Q_3)$, we see that 
${_2}\Br(Q_3)\cong{_2}\H_G^3(X,\Z(1))$
\end{proof}

Combining Lemma \ref{Br(Q3)} with Corollary \ref{WR-3fold}, we conclude:

\begin{corollary}
$W(Q_3) \cong WR(Q_{3,\top}) \cong \Z/8$.
\end{corollary}

When $V$ is a geometrically connected 3-fold over $\R$,
and $V(\R)\ne\emptyset$, we saw in Proposition \ref{exp.2^d} that
the torsion subgroup of $I(V)$ has exponent $8$.
We have not yet thought about this case.

\newpage
\section{Fundamental Theorem for Witt groups}\label{sec:Gm}
\medskip

The Fundamental Theorem for Witt groups concerns Laurent polynomial extensions.
When $V=\Spec(A)$, and $A$ is a regular $\R$-algebra, 
the formula $W(A[t,1/t])\cong W(A)\oplus W(A)$ goes back to
the work of the first author 
\cite[pp.\,138--9]{MKlocalisation} and Ranicki \cite[4.6]{Ranicki}
in the 1970s. More recently, 
Balmer and Gille proved in \cite[9.13]{BalmerGille} that for smooth schemes
$V$ we have $W(V\times\Spec(\R[t,1/t]))\cong W(V)\oplus W(V)$. 

Since the $G$-space underlying $\Gm$ is equivariantly homotopy equivalent 
to $S^{1,1}$, we consider $WR(X\times S^{1,1})$. Since 
$WR(S^{1,1})\cong\Z\oplus\Z$ (by Lemma \ref{WR(graph)}) we have a
canonical external product map
\begin{equation}\label{eq:cup}
WR(X)\oplus WR(X) \cong WR(X)\otimes WR(S^{1,1}) 
\smap{\mu} WR(X\times S^{1,1}).
\end{equation}

When $X$ is a finite $G$-set, 
it follows from Lemma \ref{WR(graph)} that 
the product \eqref{eq:cup} is an isomorphism.
However, it is not true in general that 
$WR(X\times S^{1,1})\cong WR(X)\oplus WR(X)$,
even if $X=V_\top$, as Example \ref{SGm} shows.

Here is a sufficient condition. To state it, recall that sending
a Real vector bundle $E$ to its dual $E^*$ defines an action of
$\Z/2$ on $KR(X)$, and hence on $KR_{n}(X)=KR^{-n}(X)$ for all $n$.%
\footnote{The superscript notation $KR^*$ is used by Atiyah \cite{Atiyah},
but the subscript notation is more appropriate for our uses.}
Let $kr_{n}(X)$ (resp.\ $kr'_{n}(X)$) denote the even (resp.\ odd) Tate
cohomology of $\Z/2$ acting on the group $KR_{n}(X)$.
That is,
\[
kr_n(X)= H^2(\Z/2,KR_{n}(X)), \quad kr'_n(X)=H^1(\Z/2,KR_{n}(X)).
\]

\begin{theorem}\label{F.Thm}
If $kr_{-1}(X)=kr'_{-1}(X)=0$, 
(which is true in particular when $KR_{-1}(X)=0$), then
the product \eqref{eq:cup} is an isomorphism: 
\[
 W(X)\oplus WR(X)\ \map{\mu}\ WR(X\times S^{1,1}).
\]
More generally, the kernel of $\mu$ is a quotient of $kr_{-1}(X)$
and the cokernel of $\mu$ is a subgroup of $kr'_{-1}(X)$.
\end{theorem}

We will prove Theorem \ref{F.Thm} later in this section.
We first record:

\begin{lemma}\label{KR(XxGm)}
$KR_n(X\times S^{1,1}) \cong KR_n(X) \oplus KR_{n-1}(X)$.
\end{lemma}

\begin{proof}
Since $B^{1,1}$ retracts equivariantly to a point in $S^{1,1}$,
we have a (split) short exact sequence
\[
0 \to KR_n(X\times B^{1,1}) \to\! KR_n(X\times S^{1,1}) \to\!
     KR_{n-1}(X\times B^{1,1}\!,X\times S^{1,1})\! \to\!0.
\]
Since the Thom Isomorphism theorem \cite[2.4]{Atiyah},
identifies the final term with 
$KR_{n-1}(X^E\!,\pt)\cong KR_{n-1}(X)$, $E=X\times\R^{1,1}\!,$
the result follows.
\end{proof}

\begin{subex} 
When $X=S^{1,1}$, Lemma \ref{KR(XxGm)} yields 
$KR_{-1}(X)=KR_{-1}(\pt)\oplus KR_{-2}(\pt)=0$ 
and hence $WR(X\times S^{1,1})\cong WR(X)^2\cong\Z^4.$
We also have $W(\Gm \times \Gm)\cong WR(S^{1,1} \times S^{1,1})$ in this case.
\end{subex} 

Since Real vector bundles on $X$
form a Hermitian category, there is a classical Bott exact sequence
(see \cite[Thm.\,6.1]{Schlichting.Fund}): 
\begin{equation}\label{seq:GW[i]}
\to {GR_{n}^{[i-1]}}(X) \smap{F} KR_{n}(X)%
\overset{H}{\to }{GR_{n}^{[i]}}(X)\to GR_{n-1}^{[i-1]}(X) \smap{F}.
\end{equation}
The groups $GR_0(X)=GR_n^{[0]}(X)$ are the usual Grothendieck--Witt groups
of symmetric forms on Real vector bundles (see \cite[1.39]{Schlichting.Fund}),
and the groups $GR_n^{[2]}(X)$ are the Grothendieck--Witt groups of 
skew-symmetric forms.  
\goodbreak

For the rest of this section,  
we shall adopt the following terminology. 
If $F$ is a functor on $G$-spaces, we'll write $F(\R)$ and $F(\RX)$
for $F(S^{1,1},\pt)$ and $F(X\times S^{1,1},X)$. 
Thus $KR_n(\RX)\cong KR_{n-1}(X)$ by  \ref{KR(XxGm)}, and
Theorem \ref{F.Thm} asserts that $WR(\RX)\cong WR(X)$.

We will use the following description of the group $\VR_1(X)$,
viewed as the Grothendieck group of the forgetful functor $F$
(see \cite[II.2.13]{MKbook}).
An element is determined by a triple $(E,g_1,g_2)$, up to a suitable
homotopy equivalence. Here $E$ is a Real vector bundle on $X$ and 
the $g_i$ are symmetric bilinear forms on $E$.

Taking $i=n=1$ in \eqref{seq:GW[i]}, and using Lemma \ref{KR(XxGm)},
we get an exact sequence
\begin{equation}\label{eq:GR(RX)}
GR_1(\RX) \smap{F} KR_0(X) \smap{H} \VR_1(\RX)\smap{\partial} GR(\RX) 
\smap{}KR_{-1}(X).
\end{equation}
\begin{example}\label{ex:pt-cup-u}
When $X$ is a point, the group $\VR_1(\R)$ has two special elements:
$u=(E,\langle1\rangle,\langle t\rangle)$ and 
$u^-=(E,\langle-1\rangle,\langle -t\rangle)$, where
$E$ is the bundle $S^{1,1}\times \C$ and
$\langle\pm t\rangle$ (resp., $\langle\pm1\rangle$)
is the quadratic form $x\mapsto\pm tx^2$ (resp., $x\mapsto\pm x^2$).
The boundary map $\partial$ in \eqref{eq:GR(RX)} sends $u$ to the generator 
$(E,\langle t\rangle)$ of $GR(\R)\cong\Z$, 
and $\partial(u^-)=-\partial(u)$, while 
the map $H$ in \eqref{eq:GR(RX)} 
sends the generator of $KR_0(\pt)\cong\Z$ to $u+u^-$.
Since $GR_1(\R)\cong\Z/2$, the map
$F:GR_1(\R) \to KR_0(\pt)$ is zero. Since $KR_{-1}(\pt)=0$,
\eqref{eq:GR(RX)} reduces to:
\[
 0\to KR_0(\pt) \smap{H} \VR_1(\R) \smap{\partial} GR(\R) \to 0. 
\]
It follows that $\VR_1(\R)$ is isomorphic to $\Z^2$, 
with basis $\{u,u^-\}$. 
%
%

Similarly, $GR(\pt)\cong RO(G)\cong\Z^2$; a basis is given
by $e=(\C,\langle 1\rangle)$ and $e^-=(\C,\langle -1\rangle)$.
Since the cup product $GR(\pt) \map{\cup\,u} \VR_1(\R)$
sends $e$ to $u$ and $e^-$ to $u^-$, it is an isomorphism.
\end{example}

\begin{lemma}\label{cup-lemma}
Let $\varphi_X:KO(X)\to H(X)$ be a natural transformation of functors
from finite CW complexes to abelian groups 
such that each $H(X)$ is naturally a $KO(X)$-module
and each $\varphi_X$ is a $KO(X)$-module map.  Then $\varphi_X$ is the
cup product with $t=\varphi_\pt(1)$.

If in addition $\varphi_\pt:KO(\pt)\to H(\pt)$ is an isomorphism,
and $H(X)\cong KO(X)$ for all $X$, then
$\varphi:KO^*\to H^*$ is an isomorphism of cohomology theories.
\end{lemma}

\begin{proof}
If $x\in KO(X)$ then $\varphi_X(x)=\varphi_X(x\cdot1)=\varphi_X(x)\cdot t.$
This proves the first assertion.  Now assume that 
$H(X)\cong KO(X)$ for all $X$, and that $\varphi_\pt$ is an
isomorphism; as $KO(\pt)\cong\Z$, $t=\pm1$. To see that
$\varphi_X$ is an isomorphism, we may assume that $X$ is connected.
In this case, $KO(X)\cong KO(\pt)\oplus \widetilde{KO}(X)$,
and the second factor is a nilpotent ideal in $KO(X)$.
Therefore $t=\varphi_X(1)$ is a unit of $KO(X)$, and hence
$\varphi_X:KO(X)\to KO(X)$ is an isomorphism.
Taking products with spheres and using Bott periodicity,
it follows that $\varphi_X:KO^*(X)\cong H^*(X)$ for all $X$,
i.e., $\varphi$ is an isomorphism of cohomology theories.
\end{proof}

\begin{subremark}\label{cup-KO_G}
If $G$ acts trivially on $X$, so that 
$KO_G(X)\cong KO(X)[y]/(y^2-1)$ and $\VR(\RX)$ is a $GR(X)$-module,
the lemma holds for $KO_G(X)\to \VR(\RX)$.
\end{subremark}

\begin{theorem}\label{cup-u}
The cup product with $u$ defines isomorphisms
\[ GR_n(X) \map{\cup\, u} \VR_{n+1}(\RX).\]
\end{theorem}

\begin{corollary}\label{C.6} (See \cite[C.6]{KSW})
We have an exact sequence 
\[ \to GR_1(\RX) \to KR(X) \smap{H} GR(X)\smap{} GR(\RX)\to,
 \]
where the map $GR(X)\to GR(\RX)$ is the cup product with $\partial(u)$.
\end{corollary}

\begin{proof}
Use Theorem \ref{cup-u} to replace the middle term
in \eqref{eq:GR(RX)} with $GR(X)$.
\end{proof}

\begin{proof}[Proof of Theorem \ref{cup-u}]
We first consider two extreme cases, and then the general case.

\smallskip
\noindent {\it Case 1)} 
When $X=X^G$, the sequence of $GR(X)$-modules 
\eqref{eq:GR(RX)} reduces to
\[ 0 \to KO(X) \smap{H} \VR_{1}(\RX) \smap{\partial} KO(X) \to 0. \]
We see from Example \ref{ex:pt-cup-u} and Lemma \ref{cup-lemma}
that the map $H$ is the cup product with $u+u^-$, and
that the map $\partial$ is split by the cup product
$KO(X)\to \VR_1(X)$  with $u$.
Hence $\VR(\RX)\cong GR(X)\cong KO(X)\oplus KO(X)$ as a $KO(X)$-module.

We claim that $\VR_1(\RX)\cong GR(X)$ as a $GR(X)$-module.
Now $GR(X)\cong KO(X)\otimes RO(G)$, and $RO(G)=\Z[s]/(s^2-1)$,
it suffices to check the action of the sign representation $s$.
By inspection, $s=1$ on $KO(X)$ and $s(u)=u^-$. Therefore
$H$ and $\partial$ are $RO(G)$-maps, establishing the claim.

By Remark  \ref{cup-KO_G}, 
establishes Theorem \ref{cup-u} when $X=X^G$.

\smallskip 
\noindent {\it Case 2)} 
When $X=G\times Y$, the map $GR_n(\RX)\smap{F}\!KR_n(\RX)$ 
in \eqref{eq:GR(RX)}
is the complexification map $KO_{n-1}(Y)\to KU_{n-1}(Y)$, and
\eqref{seq:GW[i]} becomes
%
\begin{equation*}\label{eq:VR}
KO_{n+2}(Y) \smap{c} KU_n(Y) \smap{H} \VR_{n+1}(\RX) \smap{} 
KO_{n+1}(Y) \smap{c} KU_{n-1}(Y).
\end{equation*}
%
This shows that $\VR_*(\RX)$ 
is a cohomology theory on $Y$. Comparing \eqref{seq:GW[i]} 
with the classical  Bott sequence 
(see \cite[III.5.18]{MKbook},  \cite[(3.4)]{Atiyah})
\[ 
KO_{n+2}(Y) \smap{c} KU_n(Y) \to KO_n(Y) \map{\partial} 
KO_{n+1}(Y) \smap{c} KU_{n-1}(Y), 
\]
we see that $\VR_{*+1}(\RX)\cong KO_*(Y)$ as $KO_*(Y)$-modules.
Therefore it suffices to show that the cup product 
$\varphi: GR_0(X)\to \VR_1(\RX)$ sends $1$ to 
a generator of the $KO(Y)$-module $\VR_1(X)\cong KO(Y)$.
By Lemma \ref{cup-lemma}, we may assume that $Y$ is a point.
In this case, the above exact 5-term sequence reduces to
\[ 
\Z/2 \map{0} \Z \map{H=2} \Z \map{\partial} \Z/2 \to 0.
\]
Since $KO(\pt)\cong\Z$ and $\VR_1(\R)\cong\Z$,
we need only show that the distinguished element 
$u=(\C,1,\langle t\rangle)$ is a generator of $\VR_1(\R)$.

For this, we first note that $\partial(u)$ is nonzero and odd because,
by construction, $\partial(u)$ is the class of 
$\langle t\rangle$ considered as a function $S^1\to KO$, and this is
the generator of $KO_1(S^1,\pt)\cong\Z/2$.
Next, we compose the map $KO(Y)\to\VR_1(\RX)$ with the discriminant 
\[
\VR_1(\RX)\smap{D} KR_1(\RX)\cong KU(Y).
\]
Since $D$ takes an element $(E,g_1,g_2)$ of $\VR_1(\RX)$
to $[g_1g_2^{-1}]$, we have $D(u)=[t]$,
which is a generator of $KR(\R)\cong KU(\pt)\cong\Z$.
Since $KU(\pt)\smap{H} KO(\pt)$ is $\Z\smap{2}\Z$,
we see that $|u|\le2$. Hence $u$ is a generator of $\VR_1(\pt)$,
establishing Theorem \ref{cup-u} when $X=G\times Y$.

\smallskip
\noindent {\it Case 3)} 
When $G$ acts freely on $X$, there is a finite cover of $X$ by
open subspaces of the form $\{G\times U_i\}$.  By case 2), the cup
product with $u$ is an isomorphism for each of these opens, and for
their interections.  It follows from the 5--lemma that the cup
product with $u$ is an isomorphism for $X$.

\smallskip
\noindent {\it Case 4)} 
For general $X$, let $T$ be an equivariant closed neighborhood of
the subcomplex $X^G$, so that $X^G$ is in the interior of $T$ and
$X^G\subset T$ is a $G$-homotopy equivalence.
Write $Y$ (resp., $A$) for the closure of $X-T$ (resp., $T\cap Y$).
Then the cup product with $u$ determines a map of Mayer--Vietoris sequences
\begin{align*}\xymatrix@C=0.69em{ 
GR_{n+1}(T\amalg Y) \ar[r]\ar[d]^{\cong} & GR_{n+1}(A) \ar[r]\ar[d]^{\cong} &
GR_n(X)\ar[r]\ar[d]^{\cup\,u} & 
GR_{n}(T\amalg Y) \ar[r]\ar[d]^{\cong} & GR_{n}(A) \ar[d]^{\cong} \\
\VR_{n+1}(T\amalg Y) \ar[r] & \VR_{n+1}(A) \ar[r] & \VR_n(A) \ar[r] &
\VR_{n}(T\amalg Y) \ar[r] & \VR_{n}(A).}
\end{align*}
The general case follows from the 5--lemma.
\end{proof}

Recall that the ``co-Witt'' group $WR{\,}'_0(X)$ is the
kernel of the map $GR_0(X)\!\smap{F}\!KR_0(X)$ in \eqref{seq:GW[i]};
it is also the image of $\VR_1(X)\!\to\!GR_0(X)$.

\begin{proof}[Proof of Theorem \ref{F.Thm}.]
The relative group $WR(X\times S^{1,1}\!,X)\!=WR(\RX)$ fits into
the usual 12--term sequence \cite[p.\,278]{MKAnnalsH},
an elementary part of which is:
\[
kr'_{0}(\RX) \smap{j} WR{\,}'_0(\RX) \smap{\beta} WR_0(\RX) 
\smap{d} kr_{0}(\RX).
\]
By Lemma \ref{KR(XxGm)}, $kr_{0}(\RX)\cong kr_{-1}(X)$ and
$kr'_{0}(\RX)\cong kr'_{-1}(X)$. If these groups are 0,
for example if $KR_{-1}(X)=0$, then
$WR{\,}'_0(\RX) \cong WR_0(\RX)$.
In general, the kernel of $\beta$ is a quotient of $kr_{-1}(X)$
and the cokernel is a subgroup of $kr'_{-1}(X)$.

On the other hand, by definition, 
Theorem \ref{cup-u}, and Corollary \ref{C.6}:
\begin{align*}
WR{\,}'_0(\RX)& = \mathrm{im}\ \VR_1(\RX)\smap{H} GR_0(\RX)\\
           & = \mathrm{im}\ GR(X)\map{\cup\, u} GR(\RX) \\
           & = GR(X)/\mathrm{im}\ KR(X) \cong WR(X).  \qedhere
\end{align*}
\end{proof}

\vfill
\newpage\appendix

\section{Equivariant cohomology}\label{sec:Bredon}

In this section, we recall basic facts about $G$-equivariant cohomology
theories, and apply them to $KO_G$ and $KR$. 
Recall from \cite[I.2]{Bredon} that 
an {\it equivariant cohomology theory} is a sequence of functors  $h^q$
on pairs of $G$-complexes satisfying
homotopy invariance, excision and the existence of long exact sequences,
depending naturally on a pair of $G$-complex.

A {\it Bredon coefficient system} $M$ for the cyclic group $G$ of order~2
is a diagram $M(\pt)\smap{a} M(G)$, together with an involution $\sigma$ 
on $M(G)$ such that $\sigma a=a$ \cite[I.4.1]{Bredon}.
Any coefficient system $M$ determines an equivariant cohomology theory 
$H^*_G(-;M)$ \cite[I.6.4]{Bredon}.
Conversely, any $G$-equivariant cohomology theory $h^*$ defines 
a family of
coefficient systems $M=h^q$, $h^q(\pt)\to h^q(G)$, 
and there are canonical maps $\eta^q:h^q(X)\to H^0_G(X;h^q)$
for all $q\in\Z$.

If $X$ is a finite dimensional $G$-CW complex, the (convergent)
Bredon spectral sequence \cite[IV.4]{Bredon} for 
an equivariant cohomology theory $h^*$ is:
\begin{equation}\label{Bredon-ss}
E_2^{p,q}=H_G^p(X;h^q)\Rightarrow h^{p+q}(X),
\end{equation}
and the canonical maps $\eta^q$ are the edge maps.

Any coefficient system $M$ defines a sheaf
$\mathcal{M}$ on $X/G$ whose stalk at $\bar{x}$ is $M(\pt)$ or $M(G)$,
depending on whether the inverse image of $\bar{x}$ in $X$ is a 
fixed point or isomorphic to $G$. The sheaf cohomology 
$H^p(X/G;\mathcal{M})$ agrees with $H^p_G(X;M)$; this alternative
definition is due to Segal \cite{Segal}.
We remark that \eqref{Bredon-ss} agrees with Segal's spectral 
sequence \cite[5.3]{Segal}.

\begin{example}\label{local-coeffs}
For any $G$-module $M$, we have the coefficient system
$M(\pt)=M^G \to M(G)=M$, as well as a local system $\mathcal{M}$ 
on $X/G$, and $H^*_G(X;M)$ is $H^*\Hom_G(C_*(X),M)$, the cohomology 
associated to the local system $\mathcal{M}$.
If $A$ is an abelian group, regarded as a trivial $G$-module, we get
the {\it constant coefficient system} $A(0)$;
since $\Hom_G(C_*(X),A)=\Hom(C_*(X/G),A)$,
$H^p_G(X;A(0))$ is the usual cohomology group $H^*(X/G,A)$.
Note that $H^p_G(X;A(0))$ is different from the 
Borel cohomology $\H^p_G(X,A)$; see Example \ref{Borel}.
\end{example}

\begin{example}\label{Borel}
Let $A$ be an abelian group. The Borel cohomology groups
$\H^p_G(X,A)$ are defined to be $H^p(\XG,A)$, where $\XG=X\times_G{EG}$.
The $\H_G^*(-,A)$ form an  equivariant cohomology theory;
the associated coefficient system has $H^q_G(\pt)=H^q(BG,A)$ and 
$H^q_G(G)=0$ for $q\ne0$. 
Since $C_*(X\times EG)$ is a chain complex of free $\Z[G]$-modules,
quasi-isomorphic to $C_*(X)$, we see that
$\H^p_G(X,A)$ is the group hypercohomology of $C_*X$ with coefficients in $A$.
\[ \H^p_G(X,A)=H^p\Hom(C_*(X_G),A)=H^p\Hom_G(C_*(X\times EG),A).\] 

More generally, if $M$ is any $G$-module we will write
$\H_G^*(X,M)$ for the group hypercohomology
\[
\H_G^*(C_*X,M)=H^p\Hom_G(C_*(X\times EG),M);
\] 
it is also an equivariant cohomology theory.
For example, the reader may use the formula that
$\Hom_\Z(A,\Z)\cong\Hom_G(A,\Z[G])$
for any $G$-module $A$ to check that
$$\H_G^*(X,\Z[G])\cong H^*(X,\Z).$$

Now suppose that $A=\Z/2$.
The terms of the spectral sequence 
\eqref{Bredon-ss} are: $E_2^{p,q}=H^p(X^G,\Z/2)$  if $q>0$, and
$E_2^{p,0}=H^p(X/G,\Z/2)$. For example, if $X=X^G$ then 
$\H^n_G(X,\Z/2) \cong \bigoplus_{p=0}^n H^p(X,\Z/2)$.

The exact sequence of low degree terms is
\[ 
0 \to  H^1(X/G,\Z/2) \to \H^1_G(X,\Z/2) \smap{\eta^1} H^0(X^G\!,\Z/2)
\smap{d_2} H^2(X/G,\Z/2).
\]
The edge map $\H^1_G(X,\Z/2)\map{\eta^1} H^0(X^G\!,\Z/2)$ in \eqref{Bredon-ss} 
need not be onto, as we see from Example \ref{delPezzo}.

The groups $\H_G^*(X,\Z/2)$ are graded modules over
$\H_G^*(\pt,\Z/2)=\Z/2[\beta]$, natural in $X$. It follows that
$\beta\in\H_G^1(\pt,\Z/2)$ acts on the Bredon spectral sequence.
For $q>0$, it sends $E_2^{p,q}\cong H^p(X^G,\Z/2)$ isomorphically
to $E_2^{p,q+1}$.  When $X=V_\top$ for a variety $V$, 
so that $H^*_\et(V,\Z/2)\cong \H_G^*(X,\Z/2)$, it follows
from \cite{CTParimala} that when $q>\dim V$ we have
$E_2^{0,q}\cong H^0(V,\cH^q)\cong(\Z/2)^\nu$.
\end{example}

\begin{examples}\label{ex:coeff}
Let $G$ be the group of order 2.\\
a) The equivariant cohomology theory $KO_G^*$ determines
coefficient systems $KO_G^q$:
\[ KO_G^q(\pt)=KO^q(\pt)\otimes RO(G) \map{+} KO_G^q(G)=KO^q(\pt). \]
Here $RO(G)\cong\Z^2$ is the real representation ring of $G$, with generators
$[\R]$ and $[\R(1)]$, and '$+$' sends $a[\R]+b[\R(1)]$ to $a+b$.

\smallskip\noindent b)
Real $K$-theory $KR^q$ determines the coefficient system $KO^q\to KU^q$:
$KR^0$ is the constant system $\Z(0)$, $KR^2$ is $0\to\Z(1)$,
$KR^4$ is $\Z\smap2\Z$,
$KR^6=KR^{-2}$ is $\Z/2\smap{0}\Z(1)$ and
$KR^7=KR^{-1}$ is $\Z/2\to0$. We also have $KR^q=0$ for $q=1,3,5$
and $KR^q=KR^{q+8}$.

\smallskip\noindent c)
The cohomology theory $h^q=\H_G^q(-,\Z(1))$ of Example \ref{Borel}
has the coefficient systems $h^0=\Z(1)$, 
$h^{q}=(\Z/2\!\to\!0)$ if $q>0$ is odd, and $h^{q}=0$ otherwise.
In particular, there is an exact sequence
\begin{align*} 
0 \to H_G^1(X;\Z(1))\smap{}\H_G^1(X,\Z(1)) &\to (\Z/2)^\nu\ \map{d_2} 
\\
H_G^2(X;\Z(1))\smap{}\H_G^2(X,\Z(1)) &
\smap{} H^1(X^G;\Z/2)\, \map{d_2}\,H_G^3(X;\Z(1)).
\end{align*} 
If $G$ acts freely on $X$, then $H_G^n(X;\Z(1))\cong\H_G^n(X,\Z(1))$
for all $n$.
\end{examples}

\begin{subremark}
If $M$ is a $G$-module, $\H_G^*(X,M)$ denotes the Borel cohomology
with coefficients in $M$ (see \ref{Borel}), while
$H_G^*(X;M)$ denotes the Bredon cohomology of the coefficient
system $(M^G\!\to\!M)$ (see \ref{local-coeffs}).
In particular, $H_G^*(X;M) \to\H_G^*(X,M)$ need not be 
an isomorphism unless $X^G=\emptyset$.
\end{subremark}

\goodbreak
\begin{lemma}\label{H^q_G}
For each $p$ and $q$, there is a natural split exact sequence
\[
0 \to H^p(X^G,KO^q) \to H^p_G(X;KO_G^q) \map{+} H^p(X/G,KO^q) \to 0.
\]
If $X/G$ is connected and $X^G$ has $\nu$ components, then
$$H^0_G(X;KO_G)\cong\Z\oplus \Z^{\nu}.$$
Moreover, $H^1_G(X;KR^{-1})\cong H^1(X^G,\Z/2)$ and
\[
H^1_G(X;KO_G^{-1})\cong H^1_G(X;KR^{-1})\oplus H^1(X/G,\Z/2).
\]
\end{lemma}

\begin{proof}
There is a natural surjection from the coefficient system $KO_G^q$
to the constant coefficient system $KO^q(0)=(KO^q\to KO^q)$; the map from
$KO_G^q(\pt)\cong KO^q\otimes R(G)$ to $KO^q(\pt)$ is addition.
It is split by the map $KO^q\to KO_G^q$ sending $a$ to 
$a\otimes[\R]$.

Let $E^q$ denote the kernel of this surjection; 
$E^q(\pt)=KO^q$ and $E^q(G)=0$, so $H^p(X;E^q)\cong H^p(X^G,KO^q)$. 
Then $KO_G^q = E^q \oplus KO^q$ as coefficient systems.
Applying $H^p_G(X;-)$ yields the exact sequence of the lemma.
The last assertions follow from the isomorphisms
$KR^0\cong\Z(0)$ and $KR^{-1}\cong E^{-1}$ of coefficient systems.
\end{proof}

\begin{subremark}\label{KR(-1)}
The image of
$H^p_G(X;KR^{-2})\to\!H^p_G(X;KO_G^{-2})$ contains the summand
$H^p(X^G,\Z/2)$ of Lemma \ref{H^q_G}, 
since $E^{-2}=(\Z/2\!\to\!0)$ is a summand of both coefficient systems
$KR^{-2}$ and $KO_G^{-2}$.
\end{subremark}

\begin{lemma}\label{HG(0-A)}
Let $\cA$ be the coefficient system $(0\!\to\!A)$, where $A$ is a
$G$-module.  Then for every $G$-CW complex $X$,
\[H^*_G(X;\cA) \cong H^*_G(X,X^G;A) \cong H^*_G(X/X^G,\pt;A). \]
\end{lemma}

\begin{proof}
Recall that 
$C^q(X;\cA)\!=\!\Hom_G(C_q(X,X^G),A)$ is the group of 
functions $f$ on the $q$-cells $\sigma$ of $X-X^G$ satisfying 
$\overline{f(\sigma)}=f(\bar{\sigma})$;
see \cite[I-14]{Bredon}.
The usual differentials on $\Hom(C_*(X,X^G),A)$ make
$\Hom_G(C_*(X,X^G),A)$ into a cochain complex, and 
$H^*_G(X;\cA)$ is the cohomology of this complex, 
i.e., the cohomology $H^*_G(X,X^G;A) \cong H^*_G(X/X^G,\pt;A)$ 
of the local system $A$. (See Example \ref{local-coeffs}.)
\end{proof}

\begin{subex} When $G$ acts trivially on $X$, $H^*_G(X;\cA)=0$.
When $G$ acts freely on $X$, $H^*_G(X;\cA)=H^*(X/G,A)$.
\end{subex}

Recall that the coefficient system $KR^2$ is $0\!\to\!\Z(1)$.

\begin{corollary}\label{KR(2)}
$H^q_G(X;KR^{+2})\cong H^q_G(X;\Z(1))$, and
\[  
H^q_G(X;KR^{-2}) \cong H^q(X^G\!,\Z/2)\oplus H^q_G(X;\Z(1)).
\]
\end{corollary}

\begin{proof}
By construction, $C^*(X^G,\Z(1)) = 0$, so $H^*_G(X^G;\Z(1))=0$.
It follows from this and Lemma \ref{HG(0-A)} that
$H^q_G(X;KR^{+2})\cong\!H^q_G(X;\Z(1))$.

The calculation of $H^q_G(X;KR^{-2})$ follows from this and the
observation in Remark \ref{KR(-1)} that the coefficient system $KR^{-2}$
is the direct sum of $(\Z/2\to0)$ and $(0\!\to\!\Z(1))$.
\end{proof}

\begin{lemma}\label{H4-iso}
If  $\dim X^G\le3$, then
\[
h^4: H^4_G(X;KR^{-4})\ \map{}\ H^4_G(X;KO_G^{-4}) \cong H^4(X/G,\Z).
\]
is a surjection. If $\dim X^G\le2$, then $h^4$ is an isomorphism.
\end{lemma}

\begin{proof}
The coefficient map $KO_G^{-4}\smap{+}\Z(0)$ of Lemma \ref{H^q_G}
induces an isomorphism
$H^4_G(X;KO_G^{-4}) \cong H^4_G(X;\Z(0))\cong H^4(X/G,\Z)$,
because $\dim X^G<4$.
From the exact sequence of coefficient systems 
\[
0\to KR^{-4}\to\Z(0)\to (\Z/2\!\to\!0)\to 0,
\]
we get an exact sequence
\[
H^3(X^G,\Z/2) \to H^4_G(X;KR^{-4}) \to H^4_G(X;\Z(0)) \to H^4(X^G,\Z/2).
\]
The last term vanshes when $\dim X^G\le3$, and the left term
vanishes when $\dim X^G\le2$, so the middle map
$H^4_G(X;KR^{-4}) \to H^4_G(X;\Z(0))$ is onto (resp., an isomorphism)
when $\dim X^G$ is at most 3 (resp., 2).
As $KR^{-4}\!\to\Z(0)$ is the composite $KR^{-4}\!\to KO_G^{-4}\to\Z(0)$,
the result follows.
\end{proof}


\begin{example}\label{S22}
Let $X$ be $S^{2,2}$\!, the 3-sphere with $X^G=S^1$.
It is not hard to show that $KO_G(S^{2,2})=\Z^2$.
Since $KR^1(B^{2,2},S^{2,2})\cong KO^1=0$ by \cite[2.3]{Atiyah}, 
we also have $KR(S^{2,2})\cong\Z$ and hence $WR(S^{2,2})\cong\Z$.

In this case, the differential $d_2: H^1(X^G,\Z/2) \to H_G^3(X,\Z/2)$ 
is an isomorphism in the Bredon spectral sequence \eqref{Bredon-ss} 
for $KO_G(X)$.  Indeed, since $X/G$ is the suspension of $\C\bP^1$, we have 
$H^p(X/G,\Z/2)=0$ for $p\ne3$ and $H^3(X/G,\Z/2)=\Z/2$.
\end{example}
%

\end{document}